\documentclass[reqno, 12pt]{amsart}

\usepackage{amsfonts, amsthm, amsmath, amssymb}
\usepackage{hyperref}
\hypersetup{colorlinks=false}

\usepackage{tabu}

\usepackage[margin=1in]{geometry}

\usepackage{helvet}

\RequirePackage{mathrsfs} \let\mathcal\mathscr

\usepackage{hyperref}
\usepackage{dsfont}
\usepackage{upgreek}
\usepackage{mathabx,yfonts}
\usepackage{enumerate}

\numberwithin{equation}{section}

\newtheorem{theorem}{Theorem}[section] 
\newtheorem{lemma}[theorem]{Lemma}
\newtheorem{proposition}[theorem]{Proposition}
\newtheorem{corollary}[theorem]{Corollary}

\newtheorem{question}{Question}[section]

\theoremstyle{definition}

 \newtheorem*{acknowledgements}{Acknowledgements}
\newtheorem{remark}[theorem]{Remark}
\newtheorem{definition}[theorem]{Definition}
 \newtheorem*{notation}{Notation}

\renewcommand{\phi}{\varphi}

\newcommand{\PP}{\mathbb{P}}

\newcommand{\FF}{\mathbb{F}}
\newcommand{\ZZ}{\mathbb{Z}}

\newcommand{\QQ}{\mathbb{Q}}

\renewcommand{\leq}{\leqslant}

\renewcommand{\geq}{\geqslant}

\renewcommand{\c}{\mathbf{c}}

\renewcommand{\u}{\mathbf{u}}

\renewcommand{\b}{\mathbf{b}}

\renewcommand{\k}{\mathbf{k}}

\renewcommand{\r}{\mathbf{r}}

\DeclareMathOperator{\Gal}{Gal}

\let\emptyset\varnothing

\DeclareSymbolFont{bbold}{U}{bbold}{m}{n}
\DeclareSymbolFontAlphabet{\mathbbold}{bbold}

\newcommand{\md}[1]{  \left(\textnormal{mod}\ #1\right)}
 
\renewcommand{\P}{\mathbb{P}}
\newcommand{\Q}{\mathbb{Q}}
\newcommand{\F}{\mathbb{F}}
\newcommand{\N}{\mathbb{N}}
\newcommand{\R}{\mathbb{R}}

\newcommand{\Z}{\mathbb{Z}}
\renewcommand{\l}{\left}

\renewcommand{\r}{\right}
\renewcommand{\b}{\mathbf}
\renewcommand{\c}{\mathcal}
\renewcommand{\epsilon}{\varepsilon}
\renewcommand{\u}{\widetilde{1}}

\renewcommand{\leq}{\leqslant}
\renewcommand{\geq}{\geqslant}
\renewcommand{\#}{\sharp}

\DeclareMathOperator*{\Osum}{\sum{}^*}

\newcommand{\beq}[2]
{
\begin{equation}
\label{#1}
{#2}
\end{equation}
}

\title[The size of the primes obstructing the existence of rational points]
{The size of the primes obstructing the existence of rational points}

\author{E. Sofos}
\address{Max Planck Institute for Mathematics \\ 
Vivatgasse 7\\
Bonn \\ 
53111 \\ 
Germany} 
\email{sofos@mpim-bonn.mpg.de}


\date{\today}

\begin{document}

\begin{abstract}  
The sequence of the primes $p$ for which a variety over $\Q$ has no $p$-adic point  
plays a fundamental role in arithmetic geometry. 
This sequence is  
deterministic, however, we prove that 
if we choose  
a typical variety from a family  
then
the 
sequence 
has
random behavior.    
We furthermore prove that this behavior is modelled by a random walk in Brownian motion. 
This has several consequences, one of them being 
the 
description of the
 finer properties of the distribution 
 of the primes in this sequence
via the Feynman--Kac formula.
\end{abstract}

\maketitle

\setcounter{tocdepth}{1}
\tableofcontents
 
\section{Introduction}\label{s:intro} 
\subsection{Primes $p$ for which typical smooth varieties have no $p$-adic point} 
\label{s:primtypic}
The first step
in checking whether a homogeneous 
Diophantine equation defined over the rational numbers has a non-trivial rational solution 
is to check whether it has 
non-trivial solutions in the
$p$-adic completions
of the rational numbers for primes $p$ of bad reduction.  
It may be the case that
the least prime   $p$ for which there is no $p$-adic solution is large 
compared to the coefficients of the equation. Therefore, 
a straightforward
computational attempt 
to prove the non-existence of a $\Q$-point via $p$-adic checks
that does not take into consideration the  probable size of these primes $p$
will fail if the running time  
is limited compared to the size of the coefficients of the equation.
There are two basic questions one can ask for the (finite) sequence 
of primes $p$ for which a typical smooth variety has no $p$-adic point:
\begin{question}\label{qu:ena} Does this deterministic sequence behave in a random way? \end{question}
\begin{question}\label{qu:dio} If the behavior is random can we describe how much it deviates from being random? 
\end{question}

Naturally, these questions cannot be answered for any
arbitrary 
variety over $\Q$, therefore, we restrict ourselves to statements 
that hold    for `almost all'
members
in general infinite collections of varieties.
Our collections of varieties take the following shape. 
Let $V$ be a smooth projective variety over $\QQ$ equipped with a dominant morphism $f: V \to \PP_\Q^n$
with geometrically integral generic fibre.
One can view $V$ as a collection of infinitely many varieties, 
each variety being given by the fibre $f^{-1}(x)$ above a point $x\in \P^n(\QQ)$.
This setting includes several situations of central importance to arithmetic geometry, see, 
for example,~\cite{MR1695791,MR870307,MR876222} and~\cite{MR1438113}.
A natural question in this context is to study the density of  fibres with  a $\Q$-rational point.
Serre~\cite{MR1075658} 
investigated this when every fibre of $f$ 
is a conic and, in an important recent work, 
Loughran and Smeets~\cite{MR3568035}
proved that $0\%$ of the fibres of $f$ have a $\Q$-rational point.
Both investigations proceeded by examining $p$-adic solubility for all primes $p$.

Associated to $f$ there is a non-negative 
number $\Delta(f)$ that depends on the geometry of the singular fibres of $f$. 
It was introduced by Loughran and Smeets~\cite[\S 1]{MR3568035}
and it will frequently resurface
throughout our work.
\begin{definition} [Loughran and Smeets]
\label{def:Delta}
	Let $f:V \to X$ be a dominant proper morphism of smooth irreducible varieties over a
	field $k$. For each (scheme-theoretic) point
	$x \in X$ with perfect residue field $\kappa(x)$,
	the absolute Galois group $\Gal(\overline{\kappa(x)}/ \kappa(x))$ 
	of the residue field acts on the irreducible
	components of $f^{-1}(x)_{\overline{\kappa(x)}}:=f^{-1}(x) \times_{\kappa(x)} \overline{\kappa(x)}$ of multiplicity $1$. 
	We choose
	some finite group $\Gamma_x$ through which this action factors. Then we define
	$$\delta_x(f) = \frac{\# \left\{ \gamma \in \Gamma_x : 
	\begin{array}{l}
		\gamma \text{ fixes an irreducible component} \\
		\text{of $f^{-1}(x)_{\overline{\kappa(x)}}$ of multiplicity } 1
	\end{array}
	\right \}}
	{\# \Gamma_x }
	$$ and $$
	\Delta(f) = \sum_{D \in X^{(1)}} ( 1 - \delta_D(f)),$$
	where $X^{(1)}$ denotes the set of codimension $1$ points of $X$.
\end{definition} 
For $x \in \PP^n(\QQ)$ we define
the function
\begin{equation} \label{def:omeganbi8}
	\omega_f(x)
	:=\#\l\{\text{primes } p:f^{-1}(x)(\Q_p)=\emptyset\r\}.
\end{equation}
Although we might have $\omega_f(x)=+\infty$ for certain $x\in \P^n(\Q)$,  
note that   the Lang--Weil estimates \cite{LW} and Hensel's lemma
guarantee that  $\omega_f(x)<+\infty$
when $f^{-1}(x)$ is geometrically integral.
Let $H$ denote the  usual Weil height on $\PP^n(\QQ)$.
The case $r=1$ of Theorems $1.3$ and $1.12$ in the work of Loughran and Sofos~\cite{arXiv:1711.08396}
implies that  
\[\limsup_{B\to+\infty}
\frac{1}{\#\{ x \in \PP^n(\QQ): H(x)\leq B, f^{-1}(x) \text{ smooth}\}}
\sum_{\substack{ x \in \PP^n(\QQ) , H(x)\leq B\\ f^{-1}(x) \text{ smooth}}} 	\omega_f(x)
 \]
is bounded
 if and only if $\Delta(f)=0$.
Put in simple terms, the condition  $\Delta(f)=0$ is equivalent to the generic 
variety $f^{-1}(x)$
having too  few primes $p$ for which there is no $p$-adic point.
One example with $\Delta(f)=0$
is given by 
\[V: \sum_{i=0}^4 x_i y_i^2=0\subset \P^4 \times \P^4\]
and $f:V\to \P^4$ defined by $f(x,y)=x$. Here, for all
$x\in \P^n(\Q)$ with $f^{-1}(x)$ smooth
we have $\omega_f(x)=0$, see~\cite[\S 4.2.2,Th.6(iv)]{serrecourse}.
To avoid such examples
we shall  
  study the statistics of the set of primes in~\eqref{def:omeganbi8}
only when  $\Delta(f)\neq 0$.

To state our results it will be convenient to use the following notation:
for all $B\geq 1$ we introduce the set 
\[
\Omega_B :=\{x\in \P^n(\Q):H(x)\leq B\}  \]
and let 
$\b P_B$ denote the uniform probability measure on $\Omega_B$, 
that is for any
$A\subseteq \P^n(\Q)$ we  let 
\[
\b P_B(A)
:=
\frac{\#\{
x\in \Omega_B : x \in A
 \}}{\#\Omega_B 
}
.\]

\begin{definition}[The $j$-th smallest obstructing prime]
\label{def:defcases}
For $x\in \P^n(\Q)$ 
and $j\in \Z \cap [0,\omega_f(x)]$
we define $p_0(x):=-\infty$
and for $j\geq 1$ we define 
$p_j(x)$ to be the $j$-th
smallest  prime $p$ such that $f^{-1}(x)$ has no $p$-adic point.
If $j >\omega_f(x)$ 
 we define 
$p_j(x):=+\infty.$
\end{definition}

\subsection{Distribution of the least obstructing prime} \label{s:leastprop} Before continuing with our discussion on the distribution of every element in the sequence $\{p_j(x)\}_{j\geq 1}$
we provide a result concerning the typical size of $p_1(x)$ .
\begin{theorem}      
\label{prop:typicalprop}
Assume that 
 $V$ is a smooth  projective variety over $\QQ$ equipped with a dominant morphism $f: V \to \PP_\Q^n$ with geometrically integral generic fibre and $\Delta(f)\neq 0$. 
Let  $\xi: \R_{>0}\to \R_{>0}$ be any function  that satisfies
$\lim_{B\to+\infty} \xi(B)=+\infty$.
Then 
\beq{eq:seedel}
{\#\{x\in \P^n(\Q):H(x)\leq B,p_1(x)>\xi(B) \}
\ll
B^{n+1} \l(\frac{\log \log \xi_0(B)}{ \log \xi_0(B)} \r)^{\Delta(f)}
,} where 
 $\xi_0(B):=\min\{B^{1/20},\xi(B)\}$.
In particular,
\[\lim_{B\to+\infty} 
\b P_B[x\in \Omega_B: p_1(x) \leq \xi(B) ]=1
.\]
\end{theorem} Thus, taking 
$\xi(B)$   tending to infinity arbitrarily slowly,
shows that the typical value of  $p_1(x)$  is nearly 
 bounded.
Furthermore, by~\eqref{eq:seedel} we see that 
the largest the value of $\Delta(f)$ becomes, 
the smallest the typical value of $p_1(x)$ is.
This might be computationally useful.

\subsection{Equidistribution of obstructing primes via moments} \label{s:equit}
Let us now move to
Question~\ref{qu:ena}.
By the case $r=1$ of~\cite[Th.1.3]{arXiv:1711.08396}
we see that for $x\in \P^n(\Q)$ with  $H(x)\leq B$ and 
$f^{-1}(x)$ smooth
the usual size of $\omega_f(x)$ is $\Delta(f) \log \log B$.
Furthermore, by Lemma~\ref{lem:boundsgf7}
we have $p_j(x)\leq B^{D+1}$
for all $j$ and for some positive $D$ that only depends on $f$.
Thus the  
points 
\beq{eq:determseq}{
\log \log p_1(x)
<
\log \log p_2(x)
<\ldots<
\log \log p_{\omega_f(x)}(x)
}
are approximately 
$\Delta(f) \log \log B$ in cardinality 
and they all
 lie in an interval whose shape is approximated by the interval $ [0,\log \log B]  $.
Therefore, if the finite 
sequence~\eqref{eq:determseq}
was equidistributed then the subset $A$
of all $x\in \P^n(\Q)$ for which 
\beq{eq:affrox}
{\log \log p_j(x) = \frac{j}{\Delta(f)} (1+o(1))
\
\text{ for all } 
1\leq j \leq \omega_f(x)
} would satisfy $\lim_{B\to+\infty}\b P_B(A)=1$.
Our first result
confirms this kind of 
equidistribution 
as long as $j$ is not taken too small.
Furthermore, it shows that 
the error in the approximation~\eqref{eq:affrox} follows a normal distribution. 
\begin{theorem} 
\label{thm:ghjk56432}
Assume that 
 $V$ is a smooth  projective variety over $\QQ$ equipped with a dominant morphism $f: V \to \PP_\Q^n$ with geometrically integral generic fibre and $\Delta(f)\neq 0$. 
Let $j:\R_{\geq 1}\to \N$ be any function with 
\[
\lim_{B\to+\infty}j(B)=+\infty
\text{ and }
\lim_{B\to+\infty}\frac{j(B)-\Delta(f)\log \log B}{\sqrt{\Delta(f)\log \log B}}=-\infty
.\]
Then for any $z \in \R$ we have 
\[
\lim_{B\to\infty}
\b P_B
\left(
x \in 
\Omega_B
:\log \log p_j(x) 
\leq 
\frac{j}{\Delta(f) }
+z\frac{j^{\frac{1}{2}}}{\Delta(f) }
\
\right)
=\frac{1}{\sqrt{2\pi}}
\int_{-\infty}^z
\mathrm{e}^{-\frac{t^2}{\!2}}
\mathrm{d}t
.\]
\end{theorem}
An analogous result for the number of distinct prime divisors of a random integer was established by
Galambos~\cite[Th.2]{MR0439795}.

One of the simplest criteria for the randomness of a sequence is equidistribution, thus 
Theorem~\ref{thm:ghjk56432} answers 
Question~\ref{qu:ena}  in an affirmative manner.
Note that the typical 
size of the $j$-th smallest prime $p$
for which the variety $f^{-1}(x)$ 
has no $p$-adic point is doubly exponential in $j$ for all large $j$, 
i.e.
 \[p_j(x) \approx \exp\l(\exp\l(\frac{j}{\Delta(f)}\r)\r) .\]
In particular, we  conclude
 that the typical size of the primes is 
independent of the variety!
Finally, we shall see in Remark~\ref{rem:neceffghd}
that the second growth assumption
placed on $j$ is necessary for Theorem~\ref{thm:ghjk56432} to hold.

Theorem~\ref{thm:ghjk56432} 
gives an approximation 
to the 
size of $p_j(x)$ for a single value of $j$, therefore, it 
is natural to ask whether the main term in the approximation 
holds for several primes  $p_j(x)$ simultaneously.
This is indeed true as our next result shows. 
\begin{theorem}
\label{thm:ghjk}
Keep the assumptions of Theorem~\ref{thm:ghjk56432}. 
Let $\epsilon>0, M>0$
be arbitrary 
and let $\xi:[1,\infty)\to [1,\infty)$ be any function such that $\lim_{B\to+ \infty}\xi(B)=+\infty$.
Then 
\[
\b P_B\l(\!
x\in \Omega_B: 
\xi(B)<j\leq \omega_f(x) \Rightarrow \Bigg|\log \log p_j(x)-\frac{j}{\Delta(f)}\Bigg|
\leq  j^{\frac{1}{2}+\epsilon} \!\r)
=1+O_{f,\epsilon,M}\!\!\l(\frac{1}{\xi(B)^M}\r)\!\!
,\]
where the implied constant depends at most on $f,\epsilon$ and $M$. 
\end{theorem}
The proof of Theorem~\ref{thm:ghjk}
is given in~\S\ref{leprosy:pulltheplug}
and it generalises an
analogous result given by Hall and  Tenenbaum~\cite[Th.10]{hallbook}
regarding the number of distinct prime divisors $\omega(m)$
of a random integer $m$.
One of the
main steps in the proof  of Theorem~\ref{thm:ghjk}
is the verification of 
Theorem~\ref{eq:bachlutesuites},
where 
moments of arbitrary order 
of  
\beq{def:truncard}{
\omega_f(x,T)	:=\#\big\{\text{primes } p\leq T:f^{-1}(x)(\Q_p)=\emptyset\big\},
\
(x\in \P^n(\Q),T\geq 1),
}
are estimated
asymptotically 
and uniformly in the parameter $T$.
The arguments behind~\cite[Th.10]{hallbook}
rely on~\cite[Th.010]{hallbook},
whose proof makes use of the fact that for every $y>0$ 
the function $y^{\omega(m)}$ is multiplicative.
The function $\omega_f(x)$ does not have this property,
which is why we have to resort to finding the moments of $\omega_f(x,T)$.

\subsection{The connection with Brownian motion}
One of the main results in the work of 
Loughran and Sofos~\cite[Th.1.2]{arXiv:1711.08396}
is that  when $\Delta(f)\neq 0$ then 
for almost all $x\in \P^n(\Q)$ we have 
\[\omega_f(x)=
\Delta(f) \log \log H(x)
+\c Z_x \sqrt{\Delta(f) \log \log H(x)}
,\] where the function $\c Z_x$ is distributed like a Gaussian random variable with mean $0$ and variance $1$,
i.e.
\[
\c Z_x\sim \c N(0,1)
.\]
One way to think of this result is as a Central Limit Theorem 
for a specific  sequence of independently distributed random events;
the probability space is to be thought as the set of all fibres $f^{-1}(x)$,
the sequence is indexed by the primes $p$
and the random event is the non-existence of $p$-adic points.
Knowing the distribution of $\omega_f$ does not provide sufficient control over the distribution of the $p_j(x)$,
which, as we already saw, corresponds to knowing the distribution of $\omega_f(x,T)$ for all $1\leq T \leq H(x)$.
Indeed, it can be shown by the second part of 
Lemma~\ref{lem:fshgtt672}
that $\omega_f(x)=\omega_f(x,H(x))+O(1)$, with an implied constant that only depends on $f$.
Thus,
$\omega_f(x)$ essentially coincides with $\omega_f(x,T)$ 
when $T$ has size $H(x)$.

The analogy with the Central Limit Theorem above is useful due to the following fact:
assume we have a sequence of independent,
identically distributed
random variables $X_i, i\geq 1$, each with mean $0$ and variance $1$. 
The Central Limit Theorem states  that the random variable 
\[
Y(n):=
\frac{1}{\sqrt{n}}
\sum_{1\leq j \leq n} X_j
\] 
 is distributed like $\c N(0,1)$ 
as $n\to+\infty$.
For every $0\leq T \leq 1$ one may also consider the averages
\[
Y(n,T):=
\frac{1}{\sqrt{n}}
\sum_{1\leq j \leq Tn} X_j
.\]
As with $\omega_f(x,T)$, we have $Y(n,T)=Y(n)$ when $T=1$.
By the Central Limit Theorem we can see
that, for fixed $T$ and as $n\to+\infty$,
$Y(n,T)$ is distributed like the normal distribution 
with mean $0$ and variance $T$ as $n\to+\infty$.
However, the random variables $Y(n,T)$ have a richer structure than $Y(n)$, namely, 
Donsker's theorem~\cite{MR0040613} 
asserts that  $Y(n,T)$ is distributed like a random walk in Brownian motion.
Brownian motion is a subject that has been widely studied throughout the last $100$ years
and, in particular, there is large collection of results regarding the distribution of these random walks.

Thus, if we showed an analogue of Donsker's theorem 
for   
 $\omega_f(x,T)$, this would enable us to use the theory of  Brownian motion 
to directly obtain  distribution theorems for the sequence of primes $p_j(x),j\geq 1$.
This is the main plan for the rest of this paper.

\subsection{Paths associated to varieties} 
\label{s:statingthmmain} 
Let $B\geq 1$ and 
$x\in \P^n(\Q)$ with $H(x)\leq B$.
It turns out that 
the appropriate object that allows to
describe the location of the primes counted by
$\omega_f(x)$ in~\eqref{def:omeganbi8}
is
$\omega_f(x,\exp(\log^t B)) $
for  $t\in [0,1]$.
Note that as  $t$ grows from $0$ to $1$, this function  grows gradually from being almost $0$  to 
becoming almost $\omega_f(x)$.
Taking $T=\exp(\log^t B)$ in Theorem~\ref{eq:bachlutesuites} 
shows that for fixed $t$ and for $B\to+\infty$
the average of this function is approximated by
\[
\Delta(f)
\log \log (\exp(\log^t B))
=
t \Delta(f)  \log \log B
.\]
This suggests 
  the following normalisation of $\omega_f(x,\exp(\log^t B))$.
\begin{definition}
\label{def:pathsdef} 
Assume that
$V$ is a smooth  projective variety over $\QQ$ equipped with a dominant morphism $f: V \to \PP_\Q^n$
with geometrically integral generic fibre and $\Delta(f)\neq 0$. 
For each  $x\in \P^n(\Q)$ and  $B\in \R_{\geq 3}$ we define the function $
X_{B}(\bullet ,x)
:[0,1]\to\R$ as follows,
\[
t\mapsto
X_{B}(t,x)
:=
\frac{\omega_f(x,\exp(\log^t B))-t \Delta(f)  \log \log B}{(\Delta(f) \log \log B
)^{\frac{1}{2}  
}} 
.\]
\end{definition}

\begin{remark}
\label{rem:pathrem}
We will later show that 
for most $x\in \P^n(\Q)$
and when $B\to+\infty$,
the function
 $X_{B}(\bullet ,x)$
 behaves like the function 
\beq{eq:behavesas}{
t\mapsto
Z_{B}(t,x)
:=
\frac{1}{(\Delta(f) \log \log B)^{\frac{1}{2}  }} 
\sum_{p\leq \exp(\log^t B)}
\begin{cases} 1-\sigma_p, &\mbox{if } f^{-1}(x)(\Q_p)=\emptyset,\\ 
-\sigma_p, & \mbox{otherwise},\end{cases}
}
where $\sigma_p$ is given by
\begin{equation} \label{def:sigma_p}
\sigma_p
:=
\frac{\#\big\{x \in \P^n(\F_{p}): f^{-1}(x) \mbox{ is non-split}\big\}}{\#\PP^n(\FF_p)}
.\end{equation}
Here, 
a scheme over a field $k$ is called split if it contains a geometrically integral open subscheme
and is called non-split otherwise. The term was introduced
by Skorobogatov \cite[Def.~0.1]{Sko96}. The weight $\sigma_p$
is $\Delta(f)/p$ on average over $p$, namely, it is shown by Loughran and Smeets~\cite[Th.1.2]{MR3568035}
that
\beq{eq:asymptolimt}
{\Delta(f)=\lim_{B\to+\infty} \frac{\sum_{p\leq B} \sigma_p}{ \sum_{p\leq B}\frac{1}{p}}  .}

For fixed $B\geq 3$ and 
$x\in \P^n(\Q)$ 
we shall show that when~\eqref{eq:behavesas} 
is thought as  a 
function of $t$, it defines  
a right-continuous
random
walk in the plane.
This random walk
moves upwards at primes $p$ for which 
the fibre $f^{-1}(x)$ has no $p$-adic point 
and moves downwards at primes $p$ for which 
the fibre
has a $p$-adic point. 
\end{remark}

Let us now recall the definition of Brownian motion from~\cite[\S 37]{MR1324786}.
First,
a \textit{stochastic process}
is collection of random variables (on a probability space $(\Omega,\c F,P)$) 
indexed by a parameter regarded as representing time.
A \textit{Brownian motion} or \textit{Wiener process}
is a stochastic process $\{B_\tau:\tau\geq 0\}$, on some 
 probability space $(\Omega,\c F,P)$,
with the following properties:
\begin{itemize}
\item The process starts at $0$ almost surely: \[P[B_0=0]=1.\]
\item The increments are independent: If $0\leq \tau_0\leq \tau_1\leq \ldots\leq \tau_k$, then for all intervals $H_i \subset \R$,
\[
P[B_{\tau_i}-B_{\tau_{i-1}} \in H_i,i\leq k]  =
\prod_{i\leq k }
P[B_{\tau_i}-B_{\tau_{i-1}} \in H_i] 
.\]
\item For $0\leq \sigma < \tau$ the increment $B_{\tau} -B_{\sigma}  $
is normally distributed with mean $0$ and variance $\tau-\sigma$, i.e. for every interval $H\subset \R$,
\[ P[B_{\tau} -B_{\sigma} \in H]=\frac{1}{\sqrt{2\pi (\tau-\sigma)}}
\int_H
\mathrm{e}^{-x^2/2(\tau-\sigma)}
\mathrm{d}x
.\]
\item For each $\omega \in \Omega$,
$B_\tau(\omega)$ is continuous in $\tau$ and 
$B_0(\omega)=0$.
\end{itemize}
Wiener showed that 
such a process
exists, see~\cite[Th.37.1]{MR1324786}.
One can thus think of $\Omega$ as the 
space of  continuous function in $[0,\infty)$ 
and $\c F$ as the $\sigma$-algebra generated by the open sets under the uniform topology in $\Omega$.

Let $D$ be the 
the space of all real-valued
right-continuous function on $[0,1]$
that have left-hand limits, see~\cite[pg.121]{MR1700749},
and consider the Skorohod topology on $D$, see~\cite[pg.123]{MR1700749}.
For any $A\subset D$ we let $\partial A:=\overline{A}\cap \overline{(D\setminus A)}$.
We denote by $\c{D}$
the Borel $\sigma$-algebra generated by the open subsets of $D$.
As explained in~\cite[pg.146]{MR1700749},
one can make 
$(D,\c{D})$ into a probability space by 
extending the classical Wiener measure from the space of continuous functions equipped
with the uniform topology to the space $D$.
This measure
will be denoted by
$W$ throughout this paper.

Note that for every $x\in \P^n(\Q)$
the function 
$X_{B}(\bullet ,x)$
is in $D$.

\begin{theorem}\label{thm:main}
Assume that  $V$ is a smooth  projective variety over $\Q$ equipped with a dominant morphism $f: V \to \P_\Q^n$
with geometrically integral generic fibre and $\Delta(f)\neq 0$. 
Let $A$ be any set   in $\c D$   with $W(\partial A)=0$.
Then 
\[
\lim_{B\to+\infty}
\b P_B
\l(
x\in \Omega_B:
 X_{B}(\bullet ,x)
 \in A\r)
=W(A)
.\]
\end{theorem}
An similar
result for strongly
additive functions defined on the integers 
was established by
Billingsley~\cite[\S 4]{MR0466055}
and
Philipp~\cite[Th.2]{MR0354602}. However,
in our situation  
the relevant 
level of distribution 
is zero, while this is not true for the analogous problem over the integers,
see Remark~\ref{rem:levofbaddiszero}.
This necessitates
the use of a truncated version of $X_B$ (see~\eqref{def:pathflat}),
which results in
more technical arguments.

Wiener's measure gives a model for Brownian motion, hence, by Remark~\ref{rem:pathrem}
our theorem has the
following 
interpretation: 
one has infinitely many random walks $X_B(\bullet,x)$ in $[0,1]\times \R$,
each walk corresponding to every fibre $f^{-1}(x)$.
The walk is
 traced out according to the existence of $p$-adic points on the variety
 $f^{-1}(x)$. 
Random walks and Brownian motion have been studied intensely in physics and probability theory,
because they provide an effective way to predict the 
 walk
traced out by a particle   in Brownian motion
according to collision with molecules. As such, the underlying 
mathematical theory needed has been particularly enriched throughout the last century,
see, for example, the book of 
Karatzas and Shreve~\cite{MR1121940}.
In the next section we shall use parts of this theory 
to provide results that go beyond Theorems~\ref{thm:ghjk56432}
and~\ref{thm:ghjk}.

\subsection{Extreme values}
We provide the first consequence 
of Theorem~\ref{thm:main}.
As one ranges over different values of $T$
the function $\omega_f(x,T)$ takes into account the finer distribution of the primes $p$
for which $f^{-1}(x)$ has no $p$-adic point. It is therefore
 important to know 
the maximal value
 of $\omega_f(x,T)$. 
This is answered by drawing upon results on the maximum value  distribution of walks in Brownian motion.
\begin{theorem}
\label{thm:beyondeq}
Keep the assumptions of Theorem~\ref{thm:ghjk56432}.
For every $z\in \R_{>0}$
we have
\beq{thm:beyondeqeq}{
\lim_{B\to+\infty}
\b P_B\l(x\in \Omega_B:
\max_{\substack{p \text{ prime } \\ p\leq B}}
\Bigg\{\frac{\omega_f(x,T)-\Delta(f) \log \log p}{(\Delta(f) \log \log H(x))^{\frac{1}{2}}}\Bigg\}
\geq z
\r)
=\frac{2}{\sqrt{2\pi}}
\int_z^{+\infty}
\mathrm{e}^{-\frac{t^2}{\!2}}
\mathrm{d}t
.}
\end{theorem}
This will turn out to be 
 a direct 
consequence of the reflection principle in Brownian motion.
Taking $p=p_j(x)$ 
in Theorem~\ref{thm:beyondeq}
leads to the following conclusion.
\begin{corollary}
\label{cor:pjcor}
Keep the assumptions of Theorem~\ref{thm:ghjk56432}.
For every $z\in \R_{>0}$ we have
\begin{align*}
\liminf_{B\to+\infty} 
\b P_B
&
\l(x\in \Omega_B:
1\leq j \leq \omega_f(x) 
\Rightarrow 
\log \log p_j(x) \geq 
\frac{j}{\Delta(f) } - z\l(\frac{\log \log H(x)}{\Delta(f)}\r)^{\frac{1}{2}}  
\r)
\\
&\geq 
1-
\frac{2}{\sqrt{2\pi}}
\int_z^{+\infty}
\mathrm{e}^{-\frac{t^2}{\!2}}
\mathrm{d}t
.\end{align*}
Furthermore, for every function $\xi(B):\R_{\geq 1}\to\R_{\geq 1}$
with 
$\lim_{B\to+\infty}\xi(B)=+\infty$
we have 
\[
\lim_{B\to+\infty} 
\! \b P_B  \!
\l(\!
x\in \Omega_B: 1\leq j \leq \omega_f(x) 
\Rightarrow 
\log \log p_j(x) \geq 
\frac{j}{\Delta(f) } - \xi(H(x))\l( \log \log H(x) \r)^{\frac{1}{2}}
\!\r)
\!=\!1.\]
\end{corollary}
In contrast to Theorem~\ref{thm:ghjk}
this result  gives merely lower bounds for
$p_j(x)$, however, it does apply to the whole range of $j$, in particular to
those that are left uncovered by Theorem~\ref{thm:ghjk}.

\subsection{Largest deviation}
Our next result 
 provides asymptotic
 estimates for the density with 
which $\omega_f(x,T)$ deviates from its expected value.
Its analogue  
in Brownian motion
regards 
 random walks 
in the presence of absorbing barriers,
see~\cite{MR0011917}.

 Let  us define the function $\tau_\infty:\R\setminus\{0\}
\to\R$ via 
\beq{def:john_lord}{
\tau_\infty(z):=\frac{4}{\pi}
\sum_{m=0}^{+\infty}   \frac{(-1)^m}{2m+1}
\exp\Bigg\{-\frac{(2m+1)^2\pi^2}{8z^2}\Bigg\}
.}
\begin{theorem}
\label{thm:beyondeqpartita}
Keep the assumptions of Theorem~\ref{thm:ghjk56432}.
For every $z\in \R_{>0}$
we have
\beq{eq:t:beyondeqpartita}{
\lim_{B\to+\infty}
\b P_B\l(x\in \Omega_B:
\max_{\substack{p \text{ prime } \\ p\leq B}}
\Bigg|\frac{\omega_f(x,p)-\Delta(f) \log \log p}{(\Delta(f) \log \log H(x))^{\frac{1}{2}}}\Bigg|
\geq z
\r)
=1-\tau_\infty(z)
.}
\end{theorem}
As with Corollary~\ref{cor:pjcor},
we have the following conclusion.
\begin{corollary}
\label{cor:pjcorpartita}
Keep the assumptions of Theorem~\ref{thm:ghjk56432}.
For every $z\in \R_{>0}$ we have
\[
\hspace{-0,1cm}
\liminf_{B\to+\infty} \b P_B
\!\l(\!  x\in \Omega_B: 1\leq j \leq \omega_f(x) 
\Rightarrow 
\Bigg|  \log \log p_j(x) -
\frac{j}{\Delta(f) } \Bigg|
\leq 
z\l(\frac{\log \log H(x)}{\Delta(f)}\r)^{\frac{1}{2}}  
\!\r) \!
\geq  \!  \tau_\infty(z)
.\]
Furthermore, the following holds 
 \[
\hspace{-0,1cm}
 \! \b P_B  \!
\l(\!
x\in \Omega_B: 1\leq j \leq \omega_f(x) 
\Rightarrow 
\Bigg|  \log \log p_j(x) -  \frac{j}{\Delta(f) } \Bigg|
\leq 
z \!\sqrt{ \log \log H(x)}
\!
\!\r)
\!\!=\!
1
+O_f\!\!\l(\frac{1}{(1+|z|)^{\frac{2}{3}  }   } \!\r)
\!,\]
with an implied constant that depends at most on $f$.
\end{corollary}
 It is
useful to compare the second limit statement in Corollary~\ref{cor:pjcorpartita} with Theorem~\ref{thm:ghjk}.
Choosing any function
$\xi(B)$  
 with 
$\xi(B)=o((\log \log B )^{\frac{1}{2}})$ in  Theorem~\ref{thm:ghjk}
will give a precise approximation for $\log \log p_j(x)$ in a range for $j$
that is wider than the range in which the
second limit statement in Corollary~\ref{cor:pjcorpartita}
gives a precise approximation. However, the advantage of Corollary~\ref{cor:pjcorpartita}
is that  it gives a better error term in the estimate for $\b P_B$
and, furthermore, it provides a better approximation to $\log \log p_j(x)$ than Theorem~\ref{thm:ghjk}
when $$(\log \log B)^{\frac{1}{2}}\ll j \leq \omega_f(x)
=
\Delta(f) (\log \log B)
(1+o(1))
.
$$

\subsection{$L_2$-norm deviations}

In statistical mechanics, the mean squared displacement (\textit{MSD})
is a `measure' of the deviation of the position of a particle with respect to a reference position over time.
One of the fundamental results of the theory of Brownian motion is that the
\textit{MSD} of a free particle during a time interval $t$ is proportional $t$.
It was studied via diffusion equations by Einstein and Langevin, see~\cite{MR3236656}.

Let us now examine an 
analogous situation
for $p$-adic solubility.
Define for $y,q\in \R$,
\[\theta_1(y,q):=2\sum_{m=0}^{\infty}(-1)^m
q^{(2m+1)^2/4}
\sin((2m+1)z)
,\] let $\theta_2(y,q):=\frac{\partial}{\partial y}\theta_1(y,q)$ 
and for $z\geq 0$ 
set 
\[\tau_2(z):=\frac{4}{\pi^{3/2}}
\int_{0\leq u \leq z/2}
\int_{0\leq t \leq \pi/2}
\theta_2(t/2,\mathrm{e}^{-1/4u})
\frac{\mathrm{d}t}{(\cos t)^{1/2}}
\frac{\mathrm{d}u}{u^{3/2}}
.\]
\begin{theorem}
\label{thm:beyondeqpartitapassacaglia}     
Keep the assumptions of Theorem~\ref{thm:ghjk56432}.
For every $z\in \R_{>0}$
we have
\beq{eq:thm:beyondeqpartitapassacaglia2}{
\lim_{B\to+\infty}
\b P_B\l(x\in \Omega_B:
\frac{1}{\l(\Delta(f) \log \log B\r)^2 }
\sum_{p\leq B}
 \sigma_p 
\Big( 
\omega_f(x,p)  -    \sum_{q\leq p  }   \sigma_q 
\Big)^2
< z
\r)
=\tau_2(z)
.}
\end{theorem}

 \subsection{Concentration of obstructing primes} 
Let us now turn our attention to
Question~\ref{qu:dio}.
The results so far show that the elements in the sequence~\eqref{eq:determseq} are equidistributed,
however, it may be that the set of primes $p$
satisfying
$f^{-1}(x)(\Q_p)=\emptyset$
is not fully equidistributed.
This could be, for example, due to
 a possible 
clustering of some of its elements.
To study the sparsity (or lack thereof) of such clusters
we shall look into the following set: 
for $x\in \P^n(\Q)$ we define 
\[ \c{C}_f(x):=\l\{p \text{ prime}: 
\omega_f(x,p)>
\sum_{q\leq p} \sigma_q
\r\}
.\]
By Lemma~\ref{lem:abelpu8}
and 
the case $r=1$ of 
Theorem~\ref{eq:bachlutesuites}
the expected value of  $\omega_f(x,p)$ is $$\Delta(f) \log \log p \approx \sum_{q\leq p} \sigma_q,$$  therefore, 
$p$ is in $\c{C}_f(x) $
exactly when there are 
`many' primes $\ell$ 
with $f^{-1}(x)(\Q_\ell)=\emptyset$ 
that are concentrated below $p$. 
Let us note that for all $x\in \P^n(\Q)$ outside a Zariski closed subset of $\P^n(x)$
this set is finite. This  is
because 
if $f^{-1}(x)$ is
smooth
then by Lemma~\ref{lem:boundsgf7} we have 
$\omega_f(x) \ll_f \frac{\log H(x)}{\log\log H(x)}$ and therefore
Lemma~\ref{lem:abelpu8}
gives 
\beq
{eq:toshth}
{
p\in \c C_f(x)
\Rightarrow
\log \log p 
\ll
\sum_{q\leq p}\sigma_p
<
 \omega_f(x,p) 
\leq 
 \omega_f(x) 
 \ll_f \frac{\log H(x)}{\log\log H(x)}.}

We wish to study the 
  distribution of $$\#\c C_f(x)=\sum_{p\in \c C_f(x) }1.$$
It turns out that it is more convenient to do so for a version of $\c C_f(x)$
where the primes are weighted appropriately.
Recall~\eqref{def:sigma_p}
and let 
\[
\widehat{\c C}_f(x)
:=\sum_{p\in \c C_f(x)} \sigma_p
.\]For 
$x\in \P^n(\Q)$ with 
$f^{-1}(x)$  smooth 
we can use~\eqref{eq:asymptolimt} to get 
$$ \widehat{\c C}_f(x)
\leq \sum_{p\leq \max\{q: \ q\in \c C_f(x)\} } \sigma_p \ll \log \log \max\{q:  q\in \c C_f(x)\}
,$$
hence, by~\eqref{eq:toshth} one has 
\beq{eq:tsabakafedaki}{
\widehat{\c C}_f(x)
 \ll_f \frac{\log H(x)}{\log \log H(x)} .}
We shall see that this bound is best possible in \S\ref{ex:klapaklapa}.

Let us now turn our attention to the average order of magnitude of $\widehat{\c C}_f(x)$.
If $\c C_f(x)$ consisted of all
 primes $p\leq H(x)$ 
then by~\eqref{eq:asymptolimt}
the
order of magnitude of
$\widehat{\c C}_f(x)$ would be 
$\log \log H(x)$. The next result shows that
there is, in fact, a distribution law for the corresponding ratio. 
\begin{theorem}
\label{thm:arcsnl}
Keep the assumptions of Theorem~\ref{thm:ghjk56432}.
For every $\alpha<\beta\in [0,1]$ we have
\beq{eq:thm:arcsnl}
{
\lim_{B\to+\infty}
\b P_B\l(x\in \Omega_B:
\frac{\widehat{\c C}_f(x)}{
 \Delta(f)   \log \log H(x)  
} \in (\alpha,\beta]\r)=
\frac{1}{\pi}
\int_\alpha^\beta
\frac{\mathrm{d}u}{\sqrt{u(1-u)}}
.}
\end{theorem}
This can be viewed as a $p$-adic solubility
analogue
of  L\'{e}vy's  arcsine law
that concerns the time that 
a random walk in Brownian motion
spends above $0$, see~\cite[\S 5.4]{MR2604525}.
One consequence of Theorem~\ref{thm:arcsnl}
is that, since  the area of  $\int_0^1(u(1-u))^{-1/2}\mathrm{d}u$
is concentrated in the regions around $u=0$ and $u=1$,
for most fibres $f^{-1}(x)$
the set of primes $p$
without a $p$-adic point 
will be 
either  very regularly
or 
very irregularly
spaced.
\subsection{The Feynman--Kac formula}
The Feynman--Kac formula plays a major role in linking stochastic processes
and partial differential equations,
see
the book of Karatzas and Schrieve~\cite[\S 4.4]{MR1121940}
and the book
of M\"orters and Peres~\cite[\S 7.4]{MR2604525}.
For its applications to other sciences see the book by Del Moral~\cite{MR2044973}.

We shall use the formula
to establish a link between $p$-adic solubility and differential equations.
Our result will roughly say that 
in situations more general than those in  Theorems~\ref{thm:beyondeq},
~\ref{thm:beyondeqpartita},
~\ref{thm:beyondeqpartitapassacaglia}  
and~\ref{thm:arcsnl}
the analogous
distributions (such as those in the right side of~\eqref{thm:beyondeqeq},\eqref{eq:t:beyondeqpartita},
~\eqref{eq:thm:beyondeqpartitapassacaglia2}
 and~\eqref{eq:thm:arcsnl})
are derived from equations similar to Schr\"{o}dinger's equation in quantum mechanics.
The following definition can be found in the work of
Kac~\cite{MR0027960}.
\begin{definition}
\label{def:scrokac}
Let $\c K
:\R\to \R_{\geq 0}$ be a non-negative
 bounded function.
For $s,u\in \R$ with $s>0$ and $u>0$
we say that a solution $\Psi_{s,u}$ 
of the differential equation 
\beq{eq:different}{
\frac{1}{2} 
\frac{d^2 \Psi_{s,u}}{d x^2}=
(s+u \c K(x)) \Psi_{s,u}(x)
}
is fundamental
if it satisfies the conditions
\begin{itemize}
\item $\lim_{|x|\to \infty}\Psi_{s,u}(x)=0$,
\item $\sup_{x\neq 0} |\Psi'_{s,u}(x)| < \infty $,
\item $\Psi'_{s,u}(+0)-\Psi'_{s,u}(-0)=-2$.
\end{itemize}
\end{definition}
Equation~\eqref{eq:different}
is related to the heat equation, see, for example, section $7.4$ in the book of M\"orters and Peres~\cite{MR2604525}.
The solution $\Psi_{s,u}(x)$ corresponds to the temperature at the place $x$ for a heat flow 
with cooling at rate $-u \c K(x)$. 

Influenced by the work of Feynman~\cite{MR0026940},
Kac~\cite{MR0027960} proved that 
a fundamental solution exists, is unique, and, furthermore, 
that for every $s>0$ and $u>0$
 it fulfils
\beq
{eq:feynmankac}
{
\int_0^{+\infty}
\mathrm{e}^{-st}
 \mathbb{E}^0
\l(
\exp\l\{-u
\int_{0}^t
\c K(B_\tau)
\mathrm{d}\tau
\r\}
\r)
\mathrm{d}t
=\int_{-\infty}^{+\infty}
\Psi_{s,u}(x)
\mathrm{d}x,}
where $ \mathbb{E}^0$ is taken over 
 all Brownian motion paths $\{B_\tau:\tau\geq 0\}$ satisfying $B_0=0$ almost surely 
 and with respect to the Wiener measure $W$.
Kac then used this to calculate
the distribution function
\[W
\l[\int_{0}^t
\c K(B_\tau)
\mathrm{d}\tau
\leq z
\r]
, (t>0,z>0),\]
for various choices of $\c K$.
Thus,~\eqref{eq:feynmankac}
employs
differential equations 
in order to allow
the use of appropriately general
``test functions'' $\c K$ that 
 measure the evolution 
through time
of the distance from the average position (i.e. $\tau=0$)
of a Brownian motion path.

Recall the meaning of 
$V,f$ and $\Delta(f)$ in~\S\ref{s:primtypic}
and the definitions of 
$\omega_f(x,T)$ and $\sigma_p$ 
in~\eqref{def:truncard} and~\eqref{def:sigma_p}
respectively.
We shall use
Theorem~\ref{thm:main}
and~\eqref{eq:feynmankac}
to study the fluctuation of $\omega_f(x,p)$ as the prime 
$p$ varies.
For this, we define for every   
 non-negative
 bounded function
 $\c K:\R\to \R_{\geq 0}$ 
and every $B\in \R_{\geq 3}$ and $t\in [0,1]$ 
the function 
$
\widetilde{\c K}_B(\bullet,t)
:\P^n(\Q)\to \R$
given by 
\beq{def:frenchsuites}{
\widetilde{\c K}_B(x,t)
:=\frac{1}{\Delta(f) \log \log B} 
\sum_{p\leq 
\exp\l(\log^t B\r)}
\sigma_p
\
\c K\!\l(
\frac{\omega_f(x,p)-\sum_{q\leq p} \sigma_q}{\sqrt{\Delta(f) \log \log B}}
\r).}
This measures how far are the values of $\omega_f(x,p)$ 
from its average when $p$ ranges;
different choices of $\c K$ correspond to different ways of measuring this distance. 
For example, 
the choices 
\[
\c K(x):=x^2
\text{ and }
\c K(x):= \b 1_{[0,\infty)}(x)
\]
are relevant to 
Theorems~\ref{thm:beyondeqpartitapassacaglia}  
and~\ref{thm:arcsnl}
respectively. Our next result allows
general non-negative bounded
``test functions'' $\c K$,
thus it provides a general method 
for dealing with 
Question~\ref{qu:dio}.

\begin{theorem}
\label{thm:feynmkac} 
Assume that 
 $V$ is a smooth  projective variety over $\QQ$ equipped with a 
dominant morphism $f: V \to \PP_\Q^n$ with geometrically integral generic fibre and $\Delta(f)\neq 0$. 
Let $\c K
:\R\to \R_{\geq 0}$ be a non-negative
 bounded function.
Then for every $u>0$ and $t\in [0,1]$
the following limit exists, 
\[
\widehat{\c K}(u,t)
:=
\lim_{B\to+\infty}
\frac{1}{\#\{x\in \P^n(\Q):H(x)\leq B\}   }
\sum_{
 x\in \P^n(\Q)
\atop
H(x)\leq B
}\exp\l(
-u
\widetilde{\c K}_B(x,t)
\r)
\]
and for every 
 $u>0$ and
$s>0$
it
satisfies 
\beq{bwv:812}{
\int_0^{+\infty}
\mathrm{e}^{-st}
\widehat{\c K}(u,t)
\mathrm{d}t
=\int_{-\infty}^{+\infty}
\Psi_{s,u}(x)
\mathrm{d}x 
,}
where $\Psi_{s,u}$ is the fundamental solution of~\eqref{eq:different}.
\end{theorem}
It is 
noteworthy
 that,
for a fixed ``test function'' $\c K,$
the  left side of~\eqref{bwv:812}
is completely determined by the 
number-theoretic data associated to the fibration
$f: V \to \PP_\Q^n$,
however,
its 
right side
is determined exclusively 
through
differential equations. 
We are not aware of previous connections
between 
the Feynman--Kac formula 
and number theory.
\begin
{acknowledgements}
We are
indebted
to Carlo Pagano for suggesting Theorem~\ref{prop:typicalprop}.
We are also
grateful to Jeremy Daniel 
for useful explanations
regarding the Feynman--Kac formula.

\end
{acknowledgements}
\begin{notation}
All implied constants in the Landau/Vinogradov notation $O(\cdot),\ll$, depend at most on the fibration $f$, except where specified by the use of a subscript.
The  counting function of the distinct prime factors  
is denoted by $\omega(m):=\#\{p \text{ prime} : p\mid m\}$
and
the standard M\"{o}bius function on the integers will be denoted by $\mu$.
\end{notation}

\section{Equidistribution}   
\label{leprosy:pulltheplug}

\subsection{Auxiliary results from number theory}
\label{s:prelim}

 \begin{lemma} [Lemma 3.1, \cite{arXiv:1711.08396}]
\label{lem:boundsgf7}
	There exists $D=D(f)$ such that if $x \in \PP^n(\QQ)$ and 
	$f^{-1}(x)$ is smooth then
	\[\omega_f(x) \ll \frac{ \log H(x) }{\log \log H(x) } \  \text{ and }  \ 
	\max \{ p : f^{-1}(x)(\QQ_p) = \emptyset\} \ll H(x)^D.\]
\end{lemma}

\begin{lemma} [Lemma 3.2, \cite{arXiv:1711.08396}]\label
{lem:fshgtt672}
	Let $g \in \ZZ[x_0,\ldots, x_n]$ be a   square-free form such that 
	$f$ is smooth away from the divisor $g(x) = 0 \subset \PP^n_\QQ$.
	Then there exists $A =A(f) > 0$ such that for all primes $p > A$ the following hold.
	\begin{enumerate}
		\item The restriction of $f$ to $\PP^n_{\FF_p}$ is smooth away from the divisor 
		$g(x) = 0 \subset \PP^n_{\FF_p}$.
		\item If $x \in \PP^n(\Q)$ and $f^{-1}(x)(\QQ_p) = \emptyset$ then $p \mid g(x)$.
	\end{enumerate}
\end{lemma}

\begin{lemma} 
[Lem.3.3, \cite{arXiv:1711.08396}]
\label{lem:harmonicgfts6}
	Let  $g$ be as in Lemma \ref{lem:fshgtt672}. Then for all primes $p$ we have 
	\[
	\sigma_p \ll \frac{1}{p}, 
	\] with an implied constant that depends at most on $f$ and $g$.
\end{lemma}

\begin{lemma} [Proposition 3.6 \cite{arXiv:1711.08396}]
\label{lem:abelpu8}
There exists a constant $\beta=\beta(f)$
such that
for all $B\geq 3$ we have  
\[
\sum_{p\leq B}
\sigma_p
=\Delta(f) (\log \log B)+
\beta_f
+O((\log B)^{-1})
.\]
\end{lemma}

\begin{lemma}  
\label{lem:consoiu}
Keep the assumptions of Theorem~\ref{thm:ghjk56432}. 
There exists $A'>0$ such that 
if $p>A'$ then $\sigma_p\leq 1/2$.
Furthermore, there exists $c=c(f,A')
\in \R_{>0}$ such that 
\[
\prod_{A'<p<T}(1-\sigma_p)^{-1}
=
c
(\log T)^{\Delta(f)}
\l(1+O\l(\frac{1}{\log T}\r)\r)
.\]
\end{lemma}
\begin{proof}By Lemma~\ref{lem:harmonicgfts6}
we have $\sigma_p\leq 1/2$ for all sufficiently large $p$.
To deal with the product in the present lemma we use a Taylor expansion to obtain
\[
\log
\prod_{A'<p<T}(1-\sigma_p)^{-1}
=
\sum_{A'<p<T} \sigma_p 
+
\sum_{k=2}^{\infty}
\frac{1}{k}
\sum_{A'<p<T}\sigma_p^{k}
.\]
By Lemma~\ref{lem:harmonicgfts6}
we can now write for all $p>A'$,
\[\sigma_p^k
\leq \sigma_p^2 2^{-k+2}
\ll p^{-2} 2^{-k}
,\] with an implied constant that is independent of $p$ and $k$.
This gives 
\[
\sum_{k=2}^{\infty}
\frac{1}{k}
\sum_{A'<p<T}\sigma_p^{k}
-
\sum_{k=2}^{\infty}
\frac{1}{k}
\sum_{A'<p }\sigma_p^{k}
\ll 
\sum_{ \substack {  k \geq 2 \\  p>T }  }
\frac{1}{ k 2^k }
\frac{1}{p^2}
\ll 
\sum_{ \substack {  k \geq 2 \\  m\in \N, m>T }  }
\frac{1}{ k 2^k }
\frac{1}{m^2}
\ll \sum_{
k \geq 2 }
\frac{1}{ k 2^k }
\frac{1}{T}
\ll \frac{1}{T}
.\]
We can now invoke Lemma~\ref{lem:abelpu8}
to obtain 
\[
\log
\prod_{A'<p<T}(1-\sigma_p)^{-1}
=
\Delta(f) \log \log T
+\beta_f+O(1/\log T)
- \sum_{p\leq A'}\sigma_p
+
\sum_{k=2}^{\infty}
\frac{1}{k}
\sum_{p>A'  }\sigma_p^{k}
+O(1/T)
.\]
Letting $c:=\mathrm{e}^{\lambda}$, where 
\[\lambda:=
\beta_f
- \sum_{p\leq A'}\sigma_p
+
\sum_{k=2}^{\infty}
\frac{1}{k}
\sum_{p>A'  }\sigma_p^{k}
,\] concludes the proof.
\end{proof}

\begin{lemma}  
[Mertens]
\label{lem:abelpu9}
There exists  $C>0$
such that
for all $B\geq 2$ we have  
\[
\sum_{p\leq B}
\frac{1}{p}
=\log \log B
+
C
+O((\log B)^{-1})
.\]
\end{lemma}

Let us now give the
main number-theoretic input for
  the succeeding sections.
They were  verified  
in the proof of~\cite[Prop.3.9]{arXiv:1711.08396}.
For a square-free integer $Q$ define 
\beq{def:econtn}{
\c A_Q:=
\#\Big\{x \in \PP^n(\Q):
H(x)\leq B, f^{-1}(x) \text{ smooth},
p\mid Q \Rightarrow f^{-1}(x)(\Q_p)=\emptyset
\Big\}
.}

\begin{lemma} 
 \label
{lem:harmonicgfts7865} 
Assume that 
 $V$ is a smooth  projective variety over $\QQ$ equipped with a dominant morphism $f: V \to \P_\Q^n$ with geometrically integral generic fibre and $\Delta(f)\neq 0$. There exist positive constants $A,\alpha,d$ that depend
 at most on $f$ and the polynomial $g$ of 
Lemma~\ref{lem:fshgtt672}
such that for a square-free integer $Q$ with 
the property 
$p\mid Q\Rightarrow 
p>A
$
and each $B>1$
we have 
\[
\Big|\c A_Q-\c A_1 \prod_{p\mid Q} \sigma_p\Big|
\ll
\frac{B^{n+1}(2\alpha d)^{\omega(Q)}}{Q \min\{p: p\mid Q\}}
+(4d)^{\omega(Q)}( Q^{2n+1} B + Q B^n(\log B)^{[1/n]} )
,\]
where
the implied constant depends at most on $f$ and $g$.
\end{lemma}

\begin{lemma} 
 \label{lem:harmonicgfts7865hjdf} 
Fix a positive integer $r$,
let
 $\c{C},\epsilon_1$ be any constants with 
 \[
\c{C}
>
\frac{3r}{2}
\ \text{ and } \
0<\epsilon_1\leq 
\min \Big\{\frac{n-1/2}{2 r (n+1) } , \frac{1}{4r}   \Big\}
\]
and define the functions 
$t_0,t_1:\R_{\geq 3}\to\R$ through 
\[
t_0(B):=(\log \log B)^\c{C} \ \text{ and } \  t_1(B)=B^{\epsilon_1}
.\] In the situation of Lemma~\ref{lem:harmonicgfts7865} 
we have 
\[
\sum_{\substack{
Q\in \N, \mu(Q)^2=1 
\\ 
\omega(Q)\leq r 
\\
p\mid Q   \Rightarrow    p   \in (t_0(B),t_1(B)]
 }} 
\Big|\c A_Q-\c A_1 \prod_{p\mid Q} \sigma_p\Big|
\ll_{\c{C},\epsilon_1,r } 
B^{n+1} (\log \log B)^{r-1-\c{C}}
,\]
where the implied constant depends at most on $f,g,\c C, \epsilon_1$ and $r$. 
\end{lemma}

\begin{remark}\label{rem:levofbaddiszero} 
Lemma~\ref{lem:harmonicgfts7865} 
may be viewed as a `level of distribution' 
result in sieve theory. The 
 main term 
$\c A_1 \prod_{p\mid Q}\sigma_p$ 
essentially behaves like $$\frac{B^{n+1}}{Q}$$ for most $Q$,
while the error term contains the expression 
\[\frac{B^{n+1} }{Q \min\{p: p\mid Q\}}
.\] Therefore, to get a power saving we need to assume that $Q$ grows at least polynomially in terms of $B$.
In sieve theory language this is phrased by 
saying that the exponent of the level of distribution is $0$.
As is surely familiar to sieve experts, 
such a bad level
of distribution 
does not allow straightforward 
applications.\end{remark}

\subsection{Proof of Theorem~\ref{prop:typicalprop}}
\label{s:lestob}    
Let us first 
recall the 
\textit{Fundamental lemma of sieve theory},
as given
in~\cite[Lem. 6.3]{iwa}. 
\begin{lemma}
\label{funda}
Let $\kappa>0$ and $y>1$. There exist two sets of real numbers 
$\Lambda^+=(\lambda_Q^+)_{Q\in \N}$ and 
$\Lambda^-=(\lambda_Q^-)_{Q\in \N}$
depending only on $\kappa$ and $y$ with the following properties:
\begin{align}
&\lambda_1^\pm=1, \label{643}   
\\
|&\lambda_Q^\pm| \leq 1 \ \text{if} \ 1<Q<y, \label{644} \\
&\lambda_Q^\pm=0 \ \text{if} \ Q\geq y,\label{645}
\end{align}
and for any integer $n>1$,
\beq{646}{\sum_{Q|n}\lambda_Q^-\leq 0 \leq \sum_{Q|n}\lambda_Q^+.}
Moreover, for any multiplicative function $g(d)$ with $0\leq g(p)<1$ and satisfying the dimension condition
\beq{647}{
\prod_{t_1\leq p <t_2}(1-g(p))^{-1}
\leq
\l(\frac{\log t_2}{\log t_1}\r)^\kappa
\l(1+\frac{K}{\log t_1}\r)}
for all $2\leq t_1 <t_2\leq y$,
 we have 
\beq{648}{\sum_{Q|P(z)}\lambda_Q^{\pm}g(Q)=
\l(1+O\!\l(\mathrm{e}
^{-{s}}\l(1+\frac{K}{\log z}\r)^{10}\r)\r)
\prod_{p<z}(1-g(p)),}
where $P(z)$ denotes the product of all primes $p<z$ and $
s=\log y/\log z
$,
the implied constant
depending only on $\kappa$.
\end{lemma} 
For the proof of Theorem~\ref{prop:typicalprop}
we can clearly assume that $\xi(B)\leq B^{1/20}$.
We then take 
\[
z_0:=(\log \xi(B))^{\Delta(f)+2\alpha d},
z:=\xi(B),
y:=B^{1/10}
,\]
where $\alpha$ and $d$ were given in Lemma~\ref{lem:harmonicgfts7865}.
We shall take $\kappa$ 
(usually referred to as the 
\textit{dimension of the sieve})
to be 
\[\kappa:=\Delta(f).\]
Letting 
\[
g(d):=\prod_{\substack{ p\mid d   \\ p>z_0}}
\sigma_p
,\]
we can use 
Lemma~\ref{lem:consoiu}
to verify~\eqref{647} in our setting.
There are three cases, according to whether 
$z_0$ is in 
$(0,t_1),[t_1,t_2)$ or $[t_2,+\infty)
$.
In the first case we
have 
\[
\prod_{t_1\leq p <t_2}(1-g(p))^{-1}
=
\prod_{t_1\leq p <t_2} 
(1-\sigma_p)^{-1}
\]
and~\eqref{647} follows directly from Lemma~\ref{lem:consoiu}.
If $z_0 \in [t_1,t_2)$ then 
\[\prod_{t_1\leq p <t_2}(1-g(p))^{-1}
=
\prod_{z_0\leq p <t_2} 
(1-\sigma_p)^{-1},\]
which, by Lemma~\ref{lem:consoiu}
equals 
\[
  \l(\frac{\log t_2}{\log z_0}\r)^{\Delta(f)} \l(1+O\l(\frac{1}{\log z_0}\r) \r)
\leq   \l(\frac{\log t_2}{\log t_1}\r)^{\Delta(f)} \l(1+O\l(\frac{1}{\log t_1}\r) \r)
.\] In the remaining case, 
$z_0 \in [t_2,+\infty)$,
we have 
\[
\prod_{t_1\leq p <t_2}(1-g(p))^{-1}
=1
,\]
which is clearly bounded by the right side of~\eqref{647}.

We now define 
for $x\in \P^n(\Q)$ the integer
  \[F_x:= \prod_{
 \substack{       p\leq B^{D+1}  \\  f^{-1}(x)(\Q_p) =\emptyset  }
} p ,\]
where  $D$ is as in  Lemma~\ref{lem:boundsgf7}
This allows us to obtain  
\begin{align*}
&\#\{x\in \P^n(\Q):H(x)\leq B,p_1(x)\geq \xi(B), f^{-1}(x) \text{ smooth}  \}
  \\ \leq & 
\#\{x\in \P^n(\Q):H(x)\leq B,\gcd(F_x,   \prod_{z_0<p<z }p)=1,
f^{-1}(x) \text{ smooth} 
   \}
  \\ =& 
\sum_{\substack{x\in \P^n(\Q) \\H(x)\leq B \\ f^{-1}(x) \text{ smooth} 
 } } \sum_{\substack{Q\in \N  \\ Q\mid \prod_{z_0<p<z }p \\ Q \mid  F_x  } }\mu(Q)
 \leq 
\sum_{\substack{x\in \P^n(\Q) \\H(x)\leq B \\
f^{-1}(x) \text{ smooth} 
 } } \sum_{\substack{Q
\in \N  \\ Q\mid \prod_{z_0<p<z }p \\ Q \mid  F_x  } }
\lambda_Q^+ 
=
 \sum_{\substack{Q\in \N  \\ Q\mid \prod_{z_0<p<z }p  } }  \lambda_Q^+ \c A_Q
,\end{align*}
where 
$\c A_Q$ was defined in~\eqref{def:econtn}
and 
we used the fact that $\mu(1)=1=\lambda_1^+$ and~\eqref{646}.
Using Lemma~\ref{lem:harmonicgfts7865} 
this becomes 
\[ \sum_{\substack{Q\in \N  \\ Q\mid \prod_{z_0<p<z }p  } }  \lambda_Q^+ 
\l(
\c A_1 \prod_{p\mid Q} \sigma_p 
+O\l(
\frac{B^{n+1}
(2\alpha d)^{\omega(Q)}
}{Q \min\{p: p\mid Q\}}
+(4d)^{\omega(Q)}( Q^{2n+1} B + Q B^n(\log B)^{[1/n]} )
\r)
\r)
.\]
Owing to  
$\c A_q\ll B^{n+1}$ and the fact that very $Q$ in the last sum is 
  square-free, we 
see that the first error term is
\begin{align*}
\ll  &
\frac{B^{n+1}}{z_0} \sum_{ Q\mid \prod_{z_0<p<z }p  }   
\frac {(2\alpha d)^{\omega(Q)}}{Q}
\leq
 \frac{B^{n+1}}{z_0}  \prod_{z_0<p<z  } \l(1+\frac{2\alpha d}{p} \r)
\\  \ll &
 \frac{B^{n+1}}{z_0}  
\l(  \frac{\log z}{\log z_0}    \r)^{2\alpha d}
=
\frac{B^{n+1}}{(\log \xi(B))^{({\Delta(f)}+  2\alpha d)}
}  
\l( \frac{\log \xi (B) }{({\Delta(f)}+  2\alpha d)\log \log \xi (B)} \r)^{2\alpha d}
\\   \ll & 
\frac{B^{n+1} }{(\log \xi(B) )^{\Delta(f)}
 }
.\end{align*}
Using the bound $(4d)^{\omega(Q)} \ll_\epsilon Q^{\epsilon}$,
valid for all $\epsilon>0$,
as well as that 
 $\lambda_Q^+$ is supported
on  $[1,y]$,
the second error term is
\[
\ll_\epsilon 
y^{\epsilon+2n+1}
B 
\sum_{Q<y} 1
\leq 
y^{\epsilon+2n+2} B 
= 
B^{ 
\frac{(\epsilon+2n+2)}{10}+1 }  
\leq 
B^{n+\frac{1}{2}}
.\]
The third error term is 
\[
\ll_\epsilon
y^{1+\epsilon}
B^n(\log B)^{[1/n]}
\sum_{Q\leq y} 1
\ll
y^{2+\epsilon}
B^{n+\epsilon }
\leq B^{n+\frac{1}{2}}
.\]
Recalling that we have assumed $\xi(B)\leq B^{1/20}$
shows 
$B^{n+1}(\log \xi(B) )^{-{\Delta(f)}} \gg B^{n+\frac{1}{2}}
$
 , hence 
the estimate 
\[\#\{x\in \P^n(\Q):H(x)\leq B,p_1(x)\geq \xi(B), f^{-1}(x) \text{ singular}  \} \ll B^{n}
\] shows that 
\[
\#\{x\in \P^n(\Q):H(x)\leq B,p_1(x)\geq \xi(B), f^{-1}(x) \text{ smooth}  \}
\ll
B^{n+1}\c J
+\frac{B^{n+1} }{(\log \xi(B) )^{\Delta(f)} }
,\] where 
\[\c J
:=
\sum_{\substack{Q\in \N  \\ Q\mid \prod_{z_0<p<z }p  } }  
\lambda_Q^+ g(Q) 
=\sum_{\substack{Q\in \N  \\ Q\mid \prod_{p<z }p  } }  
\lambda_Q^+ g(Q) .
\]
By~\eqref{648}
and
Lemma~\ref{lem:consoiu}
we see that \[\c J
\ll 
\prod_{p<z}(1-g(p))
=\prod_{z_0<p<z}(1-\sigma_p)
\ll 
\l(\frac{\log z_0}{\log z}\r)^{\Delta(f)}
\ll
\l(\frac{({\Delta(f)}+2\alpha d)  \log \log \xi(B)}{\log \xi(B)}\r)^{\Delta(f)}
,\]
therefore
\[\#\{x\in \P^n(\Q):H(x)\leq B,p_1(x)\geq \xi(B), f^{-1}(x) \text{ smooth}  \}
\ll
B^{n+1}
\l(\frac{\log \log \xi(B)}{\log \xi(B)}\r)^{\Delta(f)}
,\]
which concludes the proof.
\qed

\subsection{Equidistribution without probabilistic input}

The main object of study in this section are moments involving 
the function $\omega_f(x,T)$
that is introduced in~\eqref{def:truncard}.
For fibrations
$f$ as in~\S\ref{s:primtypic}, any $B,T\geq 3$
and for  $r \in \Z_{\geq 0}$,
the $r$-th moment is defined by 
\[
\c{M}_r(f, B,T)
:=
\sum_{\substack{ x \in \PP^n(\QQ), H(x)\leq B\\ f^{-1}(x) \text{ smooth}}}
\l(
\frac{
\omega_f(x,T)
-
\Delta(f)\log \log T
}
{\sqrt{\Delta(f)\log \log T}}
\r)^{r}.
\]

\begin{theorem} 
\label{eq:bachlutesuites}
Keep the assumptions of Theorem~\ref{thm:ghjk56432}.
Let $c$ be a fixed positive constant,
assume that $B\geq 9^{1/c}$
and let 
$T\in \R\cap[9,B^c]$.
Then for every positive integer $r$  we have 
\[\frac{\c{M}_r(f, B,T)}
{\#\{x\in \P^n(\Q):H(x)\leq B\} }
=\mu_r
+O_{f,c,r}
\Big(
B^{n+1}\frac{\log \log \log \log T}{(\log \log T)^{1/2}}
\Big)
,\]
where the implied constant depends at most on $f,c$ and $r$
but is independent of $B$ and $T$.
\end{theorem}
The restriction $T\leq B^c$ is addressed in Remark~\ref{rem:neceffghdghfdtyamel5439}.
Theorem~\ref{eq:bachlutesuites}
will be proved in~\S\ref{s:prfoftrank}.
We will then use it to verify
Theorem~\ref{thm:ghjk56432}
in~\S\ref{s:prf32789}
and 
Theorem~\ref{thm:ghjk}  
in \S\ref{s:prfkj45ghjk}. 

\subsection{Proof of Theorem~\ref{eq:bachlutesuites}}
\label{s:prfoftrank}  
For a prime $p$ we define the function 
$\theta_p:\P^n(\Q)\to\{0,1\}$ via
\beq{eq:thesol}{
\theta_p(x):=
\begin{cases} 1, &\mbox{if } f^{-1}(x)(\Q_p)=\emptyset,\\ 
0, & \mbox{otherwise}.
\end{cases}
} Let $$\epsilon_r:=\min \Big\{\frac{n-1/2}{2 r (n+1) } , \frac{1}{4r}   \Big\} .$$
First we consider the case where   \beq{eq:firstcons}{c \leq {\epsilon_r} .}
Letting 
$
T_0:=(\log \log T)^{3+3r}
$
 and 
$\omega^0_f(x,T):=\sum_{T_0<p\leq T} \theta_p(x)$ allows to  define $s(T)$ via 
\[s(T)^2:=\sum_{T_0<p\leq T} \sigma_p(1-\sigma_p)\]
It
is relatively
straightforward to modify the proof of~\cite[Prop.3.9]{arXiv:1711.08396} to show that 
\[
\sum_{\substack{ x \in \PP^n(\QQ), H(x)\leq B\\ f^{-1}(x) \text{ smooth}}}
\l(
\omega^0_f(x,T)
-
\sum_{p\in (T_0,T]
}\sigma_p
\r)^{r}
\] equals\[
\begin{cases}
c_{n}B^{n+1} \mu_r 
s(T)^r
+O_r\!
\l(B^{n+1} (\log \log T)^{\frac{r}{2} -1 }\r),  &\text{ if } 2\mid r,\\
O_r\l(B^{n+1} (\log \log T)^{\frac{r-1}{2}  }      \r), &\text{ otherwise.} 
\end{cases}
 \] 
It follows from
Lemma~\ref{lem:abelpu8}
and the definition of $T_0$
that 
\[
\sum_{T_0<p\leq T}\sigma_p
=
\Delta(f) 
\log \log T
+O_r(\log \log \log \log T)
,\]
thus, writing 
$s(T)^r= (\Delta(f) \log \log T      +O_{r}(\log \log \log \log T))^{\frac{r}{2}}$
we see that \[s(T)^r
=
(\Delta(f) \log \log T)^{\frac{r}{2}}   +O_{r}\l( (\log \log T)^{\frac{r}{2}-1} \log\log\log\log T\r) 
.\]
By Lemma~\ref{lem:fshgtt672}
there exists 
a homogeneous square-free polynomial 
$F \in \ZZ[x_0,\ldots, x_n]$ such that 
if $x \in \P^n(\Q)$
and
$f^{-1}(x)(\QQ_p) = \emptyset$ then $p \mid g(x)$.
Thus 
we can write for $x \in \P^n(\Q)$ with $g(x)\neq 0$,
\[\omega^0_f(x,T)=
 \omega_f(x,T)
+O\Big(\sum_{\substack{p\mid g  (x) \\p\leq T_0}}1\Big)
.\]
This shows that 
\beq{eq:dfer}
{
\omega_f(x,T)   -  \Delta(f) \log \log T
=\Big(
\omega^0_f(x,T)
-
\sum_{T_0<p\leq T}\sigma_p
\Big)
+O_{r}\Big(\log \log \log \log T+\sum_{\substack{p\mid f(x) \\p\leq T_0}}1\Big)
.}
It is easy to modify the proof of~\cite[Lem.3.10]{arXiv:1711.08396}
in order to show 
 that for every 
$B,z>1$,
$y\in (3,B^{\frac{1}{2(r+1)}}]$,
$m\in \Z_{\geq 0}$ and a primitive homogeneous polynomial
$G\in \Z[x_0,\ldots,x_n]$
one has 
\[
\sum_{\substack{x\in \P^n(\Q)\\H(x)\leq B\\
G(x)\neq 0
}} \Big(z +\sum_{\substack{p\mid G(x)\\p\leq y }}1\Big)^m
\ll_{F,m}
B^{n+1}(z +\log \log y )^m
\]
with an implied constant that is independent of $y$ and $z$.
Using this with~\eqref{eq:dfer}
one can prove with arguments identical to the concluding arguments 
in the proof of~\cite[Th.1.3]{arXiv:1711.08396}
that 
\beq{eq:fqp}
{
\hspace{-0,1cm}
\c{M}_r(f, B,T)=
\hspace{-0,4cm}
\sum_{\substack{ x \in \PP^n(\QQ), H(x)\leq B\\ f^{-1}(x) \text{ smooth}}}
\l(
\frac{
\omega^0_f(x,T)
-
\sum_{T_0<p\leq T}\sigma_p
}
{\sqrt{\Delta(f)\log \log T}}
\r)^{r}
+O_{r}
\l(
\frac{B^{n+1}\log \log \log \log T}{(\log \log T)^{1/2}}
\r) 
.}
We have therefore shown that 
for $T$ satisfying~\eqref{eq:firstcons}
one has 
\[
\frac{\c{M}_r(f, B,T)}
{\#\{x\in \P^n(\Q):H(x)\leq B\} }
=\mu_r
+O_{r}
\l(
\frac{B^{n+1}
\log \log \log \log T}{(\log \log T)^{1/2}}
\r)
.\]
Now assume that $c>\epsilon_r$. 
Then if $f(x)\neq 0$ we obtain 
\[\omega_f(x,T)=
\omega_f(x,B^{\epsilon_r})
+O_{c,r}(1)
\]
because $\sum_{B^{\epsilon_r}<p\leq T} \theta_p(x)\ll_r (\log T)/(\log (B^{\epsilon_r}))$,
that can be shown as in the proof of~\cite[Th.1.3]{arXiv:1711.08396}.
It is clear that we have 
$\log \log T=
(\log \log  B^{\epsilon_r})+O_r(1)$.
Noting that the set $\{x\in \P^n(\Q):H(x)\leq B, f(x)=0\}$
has cardinality
$\ll B^n$ and that if $f^{-1}(x)$ is smooth then $\omega_f(x)\ll \log H(x)$ due to Lemma~\ref{lem:boundsgf7},
we obtain that 
\[
\sum_{\substack{ x \in \PP^n(\QQ), H(x)\leq B\\ f^{-1}(x) \text{ smooth}}}
\l(
\frac{
\omega_f(x,T)
-
\Delta(f)\log \log T
}
{\sqrt{\Delta(f)\log \log T}}
\r)^{r}
\] equals \[
\sum_{\substack{ x \in \PP^n(\QQ), H(x)\leq B\\ f^{-1}(x) \text{ smooth}}}
\l(
\frac{
\omega_f(x,B^{\epsilon_r})
-
\Delta(f)\log \log B^{\epsilon_r}
}
{\sqrt{\Delta(f)\log \log B^{\epsilon_r}}}
\r)^{r}
\]
up to an error term that is
 \[
\ll B^n (\log B)^r
+
\sum_{0\leq k\leq r-1}
\
\sum_{\substack{ x \in \PP^n(\QQ), H(x)\leq B\\ f^{-1}(x) \text{ smooth}}}
{r\choose k}
\l|
\l(
\frac{
\omega_f(x,B^{\epsilon_r})
-
\Delta(f)\log \log B^{\epsilon_r}
}
{\sqrt{\Delta(f)\log \log B^{\epsilon_r}}}
\r)^{k}
\r|
 .\] 
Using~\eqref{eq:fqp} for $T=B^{\epsilon_r}$
concludes the proof of Theorem~\ref{eq:bachlutesuites}.
\qed

\begin{remark}
\label{rem:neceffghdghfdtyamel5439}
Note that 
some growth restriction on $T$ is necessary in order for Theorem~\ref{eq:bachlutesuites}
 to hold.
If, for example, it holds with $T\geq B^{\log B}$, then, 
$\log \log T \geq 2 \log \log B$, hence the average of 
$\omega_f(x,T)$ would be at least $2 \Delta(f)
 \log \log B$.
According to Lemma~\ref{lem:boundsgf7}
there exists positive constants 
$C,D$ 
that depend only on $f$
such that if $H(x)\leq B$ and $f^{-1}(x) $ 
is smooth
then 
$\omega_f(x)=\omega_f(x,C_0 B^D)$.
We also know that 
the average value of $\omega_f(x)$ is 
$\Delta(f) \log \log B$, thus one would get a contradiction because 
 $\Delta(f) \neq 0$. \end{remark}
\begin{corollary}
\label{coroaotf:gaussian} 
Keep the assumptions of Theorem~\ref{thm:ghjk56432}.
Let $c$ be a fixed positive constant,
assume that $B\geq 3^{1/c}$
and let  $T:\R_{\geq 3}\to \R_{\geq 3}$ be any function with 
\[
\lim_{B\to+\infty}T(B)=+\infty
\text{ and }
T(B)\leq B^c
\text{ for all } B\geq 1.\] 
Then for any interval $\mathcal{J}\subset \R$ we have 
\[
\lim_{B\to\infty}
\b P_B
\left(
 x \in \Omega_B:
\frac{\omega_f(x,T(B))-\Delta(f)\log \log T(B)}{\sqrt{\Delta(f)\log \log  T(B)}}
\in \mathcal{J}
 \right)
=\frac{1}{\sqrt{2\pi}}
\int_{\c{J}}
\mathrm{e}^{-\frac{t^2}{\!2}}
\mathrm{d}t
.\]
\end{corollary}
\begin{proof}
The proof  uses the moment estimates provided by Theorem~\ref{eq:bachlutesuites}
and is based on the fact that the standard normal distribution is characterised by its moments.
It is identical to the proof of~\cite[Th.1.2]{arXiv:1711.08396}
that is given in~\cite[\S 3.5]{arXiv:1711.08396}
and is thus not repeated here.
\end{proof}

\begin{remark}
\label{rem:neceffghdghfdtyamel}
Recall the definition of $D$ in 
Lemma~\ref{lem:boundsgf7}. 
The special choice
$T(B)=B^{1+D}$ of 
Corollary~\ref{coroaotf:gaussian} 
is equivalent to~\cite[Th.1.2]{arXiv:1711.08396}.
\end{remark}

\subsection{ Proof of Theorem~\ref{thm:ghjk56432}  } \label{s:prf32789}
We consider $z\in \R$ to be fixed throughout this proof. Defining
$K:\R_{\geq 3}\to\R$ via
\[K(B):=\exp\l(\exp\l(
\frac{j(B)+z\sqrt{  j(B) }}{\Delta(f) }
\r)\r)\]
makes clear that 
 \beq{eq:rtzmlpqwrk}
{j(B)=\Delta(f)\log \log K(B) - z \sqrt{j(B)}
}
and  
\[\sqrt{j(B)}=\frac{-z+\sqrt{z^2+4\Delta(f)\log \log K(B)}}{2}   .\]
This provides us with 
$\sqrt{j(B)}=\sqrt{\Delta(f)\log \log K(B)} +O_z(1)$,
which, when combined with~\eqref{eq:rtzmlpqwrk},
shows that 
\beq{eq:poiuasd}
{j(B)=\Delta(f)\log \log K(B) - z \sqrt{\Delta(f)\log \log K(B)} +O_z(1)  .}

By the assumptions of our theorem regarding $j(B)$
one can see that for all
sufficiently large $B$
the 
 inequality 
$j(B)\leq \Delta(f) \log \log B-|z|\sqrt{ \Delta(f)\log \log B}$
holds.
This shows that 
\[\frac{j(B)+z\sqrt{  j(B) }}{\Delta(f) } \leq \log \log B,\]
thus,  $K(B)\leq B$. This allows to use 
 Corollary~\ref{coroaotf:gaussian} 
with $T(B):=K(B)$
in order to obtain
\beq{eq:cortuyqhsvb}
{
\lim_{B\to+\infty}
\b P_B
 \left(  x \in \Omega_B:
\frac{\omega_f(x,K(B))-\Delta(f)\log \log K(B) }{\sqrt{\Delta(f)\log \log K(B)}  }
<
- z 
\right)
=\frac{1}{\sqrt{2\pi}}
\int_{-\infty}^{-z}
\mathrm{e}^{-\frac{t^2}{\!2}}
\mathrm{d}t
.}
For any $B,u\in \R_{\geq 3}$
and $\ell \in \N$ 
it is clear that  $
 p_\ell(x) 
>  u$ is equivalent to
$\omega_f(x,u)<\ell $.
Using this 
with  $u=K(B)$ and $\ell=j$
gives
\[\b P_B
\left(
x \in \Omega_B:\log \log p_j(x) 
>
\frac{j}{\Delta(f) }
+z\frac{\sqrt{ j}}{\Delta(f) }
\
\right)
=
\b P_B
 \left(  x \in \Omega_B:
\omega_f(x,K(B))<j(B)
\right)
,\] which, when~\eqref{eq:poiuasd} is invoked,
gives
 \[
\b P_B
 \left(  x \in \Omega_B:
\omega_f(x,K(B))<
\Delta(f)\log \log K(B) - z \sqrt{\Delta(f)\log \log K(B)} +O_z(1) 
\right)
.\]
Alluding to~\eqref{eq:cortuyqhsvb}
shows that 
\begin{align*}
\lim_{B\to+\infty}
\b P_B
\left(
x \in \Omega_B:
\log \log p_j(x) 
>
\frac{j}{\Delta(f) }
+z\frac{\sqrt{ j}}{\Delta(f) }
\
\right)
&=\frac{1}{\sqrt{2\pi}}
\int_{-\infty}^{-z}
\mathrm{e}^{-\frac{t^2}{\!2}}
\mathrm{d}t\\
&=1-\frac{1}{\sqrt{2\pi}}
\int_{-\infty}^{z}
\mathrm{e}^{-\frac{t^2}{\!2}}
\mathrm{d}t,
\end{align*}
which is clearly sufficient for Theorem~\ref{thm:ghjk56432}.
\qed
\begin{remark}
\label{rem:neceffghd}
Let us note that the 
assumption 
\[ 
\lim_{B\to+\infty}\frac{j(B)-\Delta(f)\log \log B}{\sqrt{\Delta(f)\log \log B}}=-\infty
 \]
of Theorem~\ref{thm:ghjk56432} does not allow its application  when $j$ is close to its maximal value, i.e. $\omega_f(x)$,
which, by~\cite[Th.1.2]{arXiv:1711.08396}
can be as large as 
$$\Delta(f)\log \log B
+z\sqrt{\Delta(f)\log \log B}
,$$
where  $z$ is any positive constant.
The assumption is, however, necessary.
Indeed, if we could take 
$j(B)=\Delta(f)\log \log B
+M\sqrt{\Delta(f)\log \log B}$ in Theorem~\ref{thm:ghjk56432}, where $M$ is a fixed constant,
then for every $z\in \R$
there would be infinitely many $B>3$ such that for every such $B$
there exists $x\in \P^n(\Q)$ with $H(x)\leq B$,
$f^{-1}(x)$ smooth
and $f^{-1}(x)$ having no $p$-adic point 
for some prime $p$ of size  
\[
\log \log p>
\frac{\Delta(f)\log \log B
+M\sqrt{\Delta(f)\log \log B}}{\Delta(f) }
+z\frac{\sqrt{ \Delta(f)\log \log B
+M\sqrt{\Delta(f)\log \log B}}}{\Delta(f) }
.\]
However, by
Lemma~\ref{lem:boundsgf7}
we have $\log \log p\leq \log \log B+O(1)$, therefore 
taking $M=-z+1$  gives a contradiction.
\end{remark}

\subsection{Proof of Theorem~\ref{thm:ghjk} }
 \label{s:prfkj45ghjk}
We shall use the approach in the proof of~\cite[Th.10]{hallbook},
where a similar result is proved for the number of prime divisors of an integer in place of $\omega_f$.
The approach must be altered somewhat because it is difficult to prove for $\omega_f$ 
a
statement
that is analogous to
the exponential decay bound in~\cite[Th.010]{hallbook}
which is used in the proof of~\cite[Th.10]{hallbook}, the reason being that for any $A,T>0$
the function 
$A^{\#\{ p \leq T : p\mid m\}}$ is a multiplicative function of the integer $m$, 
while this is not true for $A^{\omega_f(x,T)}$.
To prove 
Theorem~\ref{thm:ghjk}
it is clearly sufficient to restrict to the cases with 
$$\xi(B)\leq (\log \log B)^{1/2},
0<\epsilon<1/2$$ and we shall assume that both inequalities holds 
during 
the rest of the proof. By Lemma~\ref{lem:boundsgf7}
there exist $C,D>0$ that only depend on $f$ such that if $x\in \P^n(\Q)$ is such that $H(x)\leq B$ and 
$f^{-1}(x)$ has no $p$-adic point then $p\leq C B^D$. 
Fixing any $\psi>1+D$ with the property 
$C B^D\leq B^\psi$ for all $B\geq 2$ 
and letting $\chi(B):=
2\xi(B)/\Delta(f) $
we shall define 
 the set 
\[
\c A
=\Big\{x\in \P^n(\Q): 
		t \in \R\cap \big(     {\mathrm{e}}^{ {\mathrm{e}}^{\chi(B)}}
		,B^\psi \big] 
		\Rightarrow 
 |\omega_f(x,t)- \Delta(f) \log \log t |
\leq \frac{1}{2}
( \Delta(f)\log \log t)^{\frac{1}{2}+
\frac{\epsilon }{2}
} 
\Big\}
.\]
This set is well-defined because  $    {\mathrm{e}}^{ {\mathrm{e}}^{\chi(B)}}
		<B^\psi $ is implied by 
		our assumption $\xi(B)\leq (\log \log B)^{1/2}$ for all large enough $B$.
		Let us now prove that \beq{eq:skog}
{\b P_B\l( \c A\r)
 =1+O\l(\frac{1}{\xi(B)^M}\r)
.} 
Note that for this it suffices to show 
\[\b P_B\l(x\in \Omega_B:
x\in \c A,f^{-1}(x) \text{ smooth } \r)
 =1+O_{\epsilon,M}\l(\frac{1}{\xi(B)^M}\r)
\] 
because 
$\#\{x\in \P^n(\Q):f^{-1}(x) \text{ singular }, H(x)\leq B\}
\ll 
\#\{x\in \c A: H(x)\leq B\}/B$.
	For $k \in \N$ we let $t_k:={\mathrm{e} }^{\mathrm{e}^k}$ and we find the
	largest $k_0=k_0(B)$ and the smallest $k_1=k_1(B)$ so that 
	\[\big(     {\mathrm{e}}^{ {\mathrm{e}}^{\chi(B)}}
		,B^\psi \big] 
		\subseteq
		\bigcup_{k=k_0}^{k_1}   (t_k, t_{k+1}]
	.\]
	Thus we deduce that if 
	$H(x)\leq B$ is such that
	$x\notin \c A$ then there exists $k\in [k_0,k_1)$ and $t\in \R$
	having the properties 
	$t \in  (t_k, t_{k+1}]$
	and $
 |\omega_f(x,t)- \Delta(f) \log \log t |
\leq \frac{1}{2}
( \Delta(f)\log \log t)^{\frac{1}{2}+
\frac{\epsilon }{2}
} $.
The last inequality implies that 
either  \[
\omega_f(x,t_{k+1}) 
 \geq 
\omega_f(x,t) \geq  \Delta(f) \log \log t 
- \frac{1}{2}
( \Delta(f)\log \log t)^{\frac{1}{2}+
\frac{\epsilon }{2}
} 
 \geq   \Delta(f) k  - \frac{1}{2}  ( \Delta(f) (k+1))^{\frac{1}{2}+ \frac{\epsilon }{2}
} 
\]
\[ \geq   
\Delta(f) (k+1)  -  ( \Delta(f) (k+1))^{\frac{1}{2}+ \epsilon } 
\]
 or 
 \[
\omega_f(x,t_k)\leq 
 \omega_f(x,t)
\leq 
\Delta(f) \log \log t +   \frac{1}{2} ( \Delta(f)\log \log t)^{\frac{1}{2}+ \frac{\epsilon }{2}} 
\leq 
\Delta(f) (k+1) +   \frac{1}{2} ( \Delta(f)  (k+1) )^{\frac{1}{2}+ \frac{\epsilon }{2}} 
\]
\[\leq 
\Delta(f) k +    ( \Delta(f)  k  )^{\frac{1}{2}+  \epsilon  } 
.\]Letting $\ell$ denote $k+1$ or $k$ respectively, we have shown that 
  the cardinality of  $x\notin \c A$
with $f^{-1}(x)$ smooth  
is at most 
\[
\sum_{\substack{ \ell \in \N \\  k_0   \leq \ell  \leq 1 + k_1  } }
\#\{x\in \P^n(\Q):
 f^{-1}(x) \text{  smooth  },
H(x)\leq B,
|\omega_f(x,t_\ell)-\Delta(f) \ell  | > (\Delta(f) \ell  )^{\frac{1}{2}+\epsilon}
\}
.\]
Note that the inequalities 
$t_{1+k_0}>   {\mathrm{e}}^{ {\mathrm{e}}^{\chi(B)}}$ and 
$t_{k_1} \leq B^\psi$ imply  that 
$k_0>-1+\chi(B)$ 
and $t_{1+k_1}= t_{k_1}^{\mathrm{e}}  
\leq B^{\mathrm{e} \psi}$.
Therefore the sum above is at most  
\[
\sum_{\substack{ \ell \in \N \\  -1+\chi(B)  <  \ell  \leq 1 + (\log \psi) +(\log \log B)  } }
\hspace{-1.6cm}
\#\{x\in \P^n(\Q):
 f^{-1}(x) \text{  smooth  },
H(x)\leq B,
|\omega_f(x,t_\ell)-\Delta(f) \ell  | > (\Delta(f) \ell  )^{\frac{1}{2}+\epsilon}
\}
.\]
Letting 
$m=m(\epsilon) $ be the least integer with $2m\epsilon \geq M+1$ 
and using 
  Chebychev's inequality  we  see that the sum  is at most 
\[
\sum_{\substack{ \ell \in \N \\  -1+\chi(B)  <  \ell  \leq 1 + (\log \psi) +(\log \log B)  } }
\frac{1}{(\Delta(f)\ell)^{2m \epsilon }} 
\sum_{\substack{ x \in \PP^n(\QQ), H(x)\leq B\\ f^{-1}(x) \text{ smooth}}}
\l(
\frac{
\omega_f(x,t_\ell)
-
\Delta(f)\ell
}
{\sqrt{\Delta(f)\ell}}
\r)^{2m}
.\]
Let us now 
apply Theorem~\ref{eq:bachlutesuites} with 
$r=2m,
c=\mathrm{e} \psi
$
and 
$T=t_\ell \leq B^{\mathrm{e}  \psi}=B^c$.
We obtain that the expression above is 
\[
\ll_{m,\psi} 
\sum_{\substack{ \ell >   -1+\chi(B)    } }
\frac{\#\{x\in \P^n(\Q):H(x)\leq B\}}{\ell ^{2M\epsilon  }}  
,\]
which is 
$O(\#\{x\in \P^n(\Q):H(x)\leq B\}\xi(B)^{-M})$ because 
$2M\epsilon \geq M+1$. This concludes the proof of~\eqref{eq:skog}.

As a last step in our proof we shall deduce
Theorem~\ref{thm:ghjk} from~\eqref{eq:skog}. 
Setting 
$t=p_j(x)$ in~\eqref{eq:skog}
shows that for  all $x\in \P^n(\Q)$ with $H(x)\leq B$, except at most $\ll B^{n+1}/\psi(B)^{M}$, 
one has 
\[
 {\mathrm{e}}^{{\mathrm{e}}^{\chi(B)}}     < p_j(x)
\leq B^\psi 
		\Rightarrow 
|
j
 - \Delta(f)  \log \log p_j(x)  |
\leq  \frac{1}{2} ( \Delta(f) \log \log p_j(x))^{\frac{1}{2}+ \frac{\epsilon }{2} } 
.\]
Recalling that $\epsilon<1/2$
the last inequality implies that 
$\Delta(f) \log \log p_j(x)  \leq 2  j  $.
Therefore the inequality 
$ {\mathrm{e}}^{{\mathrm{e}}^{\chi(B)}} < p_j(x) $ implies that 
$$\frac{2 \xi(B)}{ \Delta(f)}
=\chi(B) < \log \log p_j(x) \leq \frac{2 j}{\Delta(f)}
,$$ hence $\xi(B)\leq j$.
Finally, by 
the definition of $\psi$ we have  that the inequality 
$p_j(x)
\leq B^\psi $ is equivalent to $p_j(x)\leq \omega_f(x)$.
Owing to 
$\Delta(f) \log \log p_j(x)  \leq 2  j  $ one can see that for all sufficiently large $B$ and all $j\geq \xi(B)$
one has 
\[\frac{1}{2} ( \Delta(f) \log \log p_j(x))^{\frac{1}{2}+ \frac{\epsilon }{2} } 
\leq 
\frac{1}{2} ( 2j)^{\frac{1}{2}+ \frac{\epsilon }{2} } 
\leq 
\Delta(f)
j^{\frac{1}{2}+  \epsilon  } 
.\]
This shows that for all $x\in \P^n(\Q)$ with $H(x)\leq B$, except at most $\ll B^{n+1}/\psi(B)^{M}$, 
one has 
\[
\xi(B) < j \leq \omega_f(x) 
		\Rightarrow 
 \Big|
\log \log p_j(x) -\frac{j}{\Delta(f)} \Big|
< j^{\frac{1}{2}+\frac{\epsilon}{2}} 
,\]
thereby finishing the proof of Theorem~\ref{thm:ghjk}.
\qed

\section{Modelling by Brownian motion}
\label{s:browning_motion}
The main result in this section is Theorem~\ref{thm:main}, which proves that certain paths related to the sequence~\eqref{eq:determseq}
are distributed according to Brownian motion. To prove Theorem~\ref{thm:main} we 
begin by proving Theorem~\ref{prop:epi} in
\S\ref{s:intervcorr}. It is a generalisation of the work of
Granville and Soundararajan~\cite{granvsound}
that allows to estimate correlations that are more involved than the moments in
Theorem~\ref{eq:bachlutesuites}. We use Theorem~\ref{prop:epi} 
to verify Proposition~\ref{prop:pointwise} in~\S\ref{s:pointwiseconv}
and Proposition~\ref{prop:tightness}  in~\S\ref{s:tightness}.
These two propositions are then combined in~\S\ref{s:proofthm}
to prove Theorem~\ref{thm:main}.

\subsection{An extension of work by Granville and Soundararajan}
\label{s:intervcorr}

Assume that we are given a finite set $\c{A}$
and that
 for each $a\in \c A$
we are given
a sequence of real numbers  $\{c_n(a)\}_{n\in \N}$ 
with the property that  $\sum_{n=1}^\infty c_{n}(a)$ converges absolutely
for every $a\in \c A$.
A
central
object of study in analytic number theory 
are the 
moments
\beq{eq:agentorange}{
\sum_{a \in \c{A}}
\Big(\sum_{n \in J}c_n(a)\Big)^{k}
, k\in \N,}
where $J\subset \R$ is an interval.
In this paper 
we shall need the following generalisation.
\begin{definition}
[Interval correlation]
\label{def:interval correlation} 
Let $\c A$ and $\{c_n(a)\}_{a\in \c A}$
be as above and assume that
$J_1,\ldots, J_m\subset \R$ 
are $m$ 
pairwise disjoint
intervals.
For $\b k \in \N^m$ the $\b k$-th
interval
correlation  
of the sequence 
$\{c_n(a):n\in \N, a\in \c A \}$ 
is defined as  
\beq{eq:lefthandpath2}{
\sum_{a \in \c{A}}
\Big(\sum_{n \in J_1}c_n(a)\Big)^{k_1}
\cdots 
\Big(\sum_{n \in J_m}c_n(a)\Big)^{k_m}
.}
\end{definition}
These moments record 
how the values of $c_n(a)$ for $n$ in an interval
affect the values of $c_n(a)$ for $n$ in a different 
 interval.

The work of Granville and Soundararajan~\cite[Prop.3]{granvsound}
provides accurate estimates for the moments in~\eqref{eq:agentorange}
when the sequence 
$\{c_n(a):n\in \N, a\in \c A \}$ 
has a specific number-theoretic structure and our aim in this section is to 
use their method to provide estimates for the interval correlations in~\eqref{eq:lefthandpath2}.

Assume that
$\c{P}$ is
a finite set of primes and that $\c{A}:=\{a_1,\ldots,a_y\}$ is a multiset of $y$ natural numbers. 
For $Q\in \N$ let 
$\c{A}_Q:=\#\{m\leq y:Q \mid a_m\}$
and let $h$ be a real-valued, non-negative multiplicative function  such that 
for square-free $Q$ we have $0\leq h(Q)\leq Q$.
Whenever a square-free positive integer $Q$ satisfies 
$p\mid Q\Rightarrow p\in \c P$ we define 
 \[\c{W}(Q):=
\#\c{A}_Q-\frac{h(Q)}{Q}y
\]
and for  any $\c{P}_i\subset \c{P}$ for $1\leq i \leq m$ and 
$\b{k}\in \N^m$
we let 
\beq{eq:letusnfridcof}{
\c{E}_{\c{P}_1,\ldots,\c{P}_m}(\c{A},h,\b{k})
:=
\sum_{\substack{\b{Q}\in \N^m
\\  \forall i: \omega(Q_i)\leq k_i 
\\ \forall i: p\mid Q_i\Rightarrow p\in \c{P}_i
}} \big| \c{W}(Q_1\cdots Q_m)   \big|
\prod_{i=1}^m \mu(Q_i)^2
.}
Note that, setting  $Q:=Q_1\cdots Q_m$ provides us with
 \beq{eq:letusnfridcofbnfgh}{
\c{E}_{\c{P}_1,\ldots,\c{P}_m}(\c{A},h,\b{k})
\leq 
\sum_{\substack{
Q\in \N
\\  \omega(Q)\leq k_1+\cdots+k_m 
\\  p\mid Q \Rightarrow p\in \c{P} 
}} 
 \mu(Q)^2
\big| \c{W}(Q )   \big|
. }
Furthermore, 
for any $r\in \N$ we let 
$C_r:=\Gamma(r+1)/(2^{r/2}\Gamma(1+r/2))$, 
where $\Gamma$ is the Euler gamma function.
For any $\c{R}\subset \c{P}$ we 
define
\beq{def:themusandthesigmas}{\mu_{\c{R}}:=\sum_{p\in \c{R}} \frac{h(p)}{p},
\
\sigma_{\c{R}}:=\bigg(\sum_{p\in \c{R}} \frac{h(p)}{p}\Big(1-\frac{h(p)}{p}\Big)\bigg)^{1/2}}
and for $a\in \c A$ we define 
$
\omega_{\c{R}}(a):=\#\big\{ p\in \c{R}: p\mid a \big\}
$.
\begin{theorem} 
\label{prop:epi}
Assume that 
$\c{P}_1,\ldots,\c{P}_m$ are disjoint subsets of $\c{P}$.
Then for any $\b{k}\in \N^m$ with
$k_i\leq \sigma_{\!\c{P}_i}^{2/3}$ for all $1\leq i\leq m$
we have 
\beq{eq:timeforlunch1}{
\sum_{a\in \c{A}}
\prod_{i=1}^m
(\omega_{\c{P}_i}(a)
-\mu_{\c{P}_i}
)^{k_i}
=y
\prod_{i=1}^m
\Big(C_{k_i} \sigma_{\c{P}_i}^{k_i}
\Big(1+O\Big(\frac{k_i^3}{ \sigma_{\c{P}_i}^{2}}\Big)\Big)
\Big)
+
O\Big(
\frac{
\c{E}_{\c{P}_1,\ldots,\c{P}_m}(\c{A},h,\b{k})}
{
\prod_{i=1}^m
(1+\mu_{\c{P}_i})^{-1}
}
\Big)
}
if  $k_i$ is even for every $1\leq i \leq m$, and 
\beq{eq:timeforlunch2}{
\sum_{a\in \c{A}}
\prod_{i=1}^m
(\omega_{\c{P}_i}(a)
-\mu_{\c{P}_i}
)^{k_i}
\ll
y
\Bigg(
\prod_{i=1}^m
C_{k_i} \sigma_{\c{P}_i}^{k_i}\Bigg)
\Bigg(
\prod_{\substack{1\leq i \leq m \\k_i \text{ odd}}}
\frac{k_i^{3/2}}{\sigma_{\c{P}_i}}
\Bigg)+
\frac{
\c{E}_{\c{P}_1,\ldots,\c{P}_m}(\c{A},h,\b{k})}
{
\prod_{i=1}^m
(1+\mu_{\c{P}_i})^{-1}
}
}
if there exists $1\leq i \leq m$ such that $k_i$ is odd.
The implied constants depend at most on $m$.
\end{theorem}
\begin{proof}
As in the proof of~\cite[Prop. 3]{granvsound}
we can write 
\begin{equation} \label{eq:dodeka}
\sum_{a\in \c{A}}
\prod_{i=1}^m
\big(\omega_{\c{P}_i}(a)
-\mu_{\c{P}_i}\big)^{k_i}
=
\sum_{\substack{\forall i:
p_{1,i},\cdots,p_{k_i,i} \in \c{P}_i
}}
\sum_{a\in \c{A}} 
\prod_{i=1}^m f_{r_i}(a) 
,\end{equation}
where $r_i:=\prod_{1\leq j \leq k_i }p_{j,i}$ and $$f_r(a):=\prod_{p\mid r} 
\begin{cases}
1-\frac{h(p)}{p}, &\text{ if } p\mid a,\\
-\frac{h(p)}{p}, &\text{ otherwise.} 
\end{cases}
$$ 
Since $\c{P}_j \cap \c{P}_{j'}=\emptyset$ whenever $j\neq j'$,
we have 
$\gcd(r_j,r_{j'})=1$ for $j\neq j'$.
This allows us to 
write 
$
\prod_{i=1}^m f_{r_i}(a)=
f_{r_1\cdots r_m}(a)
$.
This allows us
to employ the estimate~\cite[Eq.(13)]{granvsound}, which provides us with 
\begin{equation} \label{eq:dekatria}
\sum_{a\in \c{A}} 
\prod_{i=1}^m f_{r_i}(a)
=yG(r_1\cdots r_m)+
\sum_{t\mid \mathrm{rad}(r_1\cdots r_m)}
\c{W}(r_1\cdots r_m)
E(r_1\cdots r_m,t)
,\end{equation}
where the entities $G,E$ are introduced in~\cite[Eq.(14)-Eq.(15)]{granvsound} 
through
\[G(r):=\prod_{p\mid r}
\Bigg(    \frac{h(p)}{p}   \Big(1-\frac{h(p)}{p}  \Big)^{\nu_p(r)}  
+  
\Big(\frac{-h(p)}{p}  \Big)^{\nu_p(r)}   
  \Big(1-\frac{h(p)}{p}  \Big)  
\Bigg)
  \] and for 
$r,t \in \N$ with  $t \mid \mathrm{rad}(r)$,
\[E(r,t):=\prod_{\substack{ p\mid r \\ p\mid t}}
\Bigg(     \Big(1-\frac{h(p)}{p}  \Big)^{\nu_p(r)}  
-
\Big(\frac{-h(p)}{p}  \Big)^{\nu_p(r)}   
\Bigg)\prod_{\substack{ p\mid r \\ p\mid \mathrm{rad}(r)/t}}
\Big(\frac{-h(p)}{p}  \Big)^{\nu_p(r)}   
.\]
The function $G$ is multiplicative, therefore using   that 
the
$r_i$ are coprime in pairs 
it is evident 
that 
the contribution of the $G$-term in~\eqref{eq:dekatria}
towards~\eqref{eq:dodeka}
is
\[\prod_{i=1}^m\Bigg(
\sum_{\substack{ 
p_{1,i},\ldots,p_{k_i,i} \in \c{P}_i
}}
G(p_{1,i} \cdots p_{k_i,i})
\Bigg)
.\] As 
shown in~\cite[pg.22]{granvsound},
one has the following estimate whenever $k\leq \sigma_{\c{P}_i}^{2/3}$,
\[
\sum_{\substack{ 
p_{1},\ldots,p_{k} \in \c{P}_i
}}
G(p_{1} \cdots p_{k })
=
\begin{cases}
C_k \sigma_{\c{P}_i}^{k} (1+O(k^3 \sigma_{\c{P}_i}^{-2})), &\text{ if } 2 \mid k,\\
O(C_k \sigma_{\c{P}_i}^{k-1}k^{3/2}), &\text{ otherwise,} 
\end{cases} 
\]
which concludes the analysis of the main term in our proposition.

It remains to study the 
contribution of the sum over $t$ in~\eqref{eq:dekatria}
towards~\eqref{eq:dodeka} and for this 
we first use the coprimality of $r_i$
to 
rewrite it as 
\[
\sum_{\substack{
\b{t}\in \N^m
\\
\forall i:
t_i\mid \mathrm{rad}(r_i)
}}
\c{W}(r_1\cdots r_m)
E(r_1\cdots r_m,t_1\cdots t_m)
.\]
We then  use the obvious
estimate $|E(r,t)|\leq \prod_{p\nmid t} h(p)/p$
to see that the said contribution is 
\[
\sum_{\substack{\boldsymbol{\ell} \in \N^m \\
\forall i:
1\leq \ell_i \leq k_i
}}
\sum_{\substack{\b{t} \in \N^m \\
\forall i:
t_i=q_{1,i}\cdots q_{\ell_i,i}
\\
q_{1,i}  < q_{2,i}  < \cdots  < q_{\ell_i,i} \in \c{P}_i
}}
|\c{W}(t_1\cdots t_m)|
\sum_{\substack{
\forall i:
p_{1,i},\cdots,p_{k_i,i} \in \c{P}_i
\\
\forall i:
t_i \mid p_{1,i} \cdots p_{k_i,i}
}}
\prod_{\substack{ 1\leq j \leq  k_i
\\ p_{j,i} \nmid t_i }} \frac{h(p_j)}{p_j}
.\]
The proof is then
concluded by alluding to the estimate
\[
\sum_{\substack{
p_{1},\cdots,p_{k} \in \c{P}_i
\\
t \mid p_{1} \cdots p_{k}
}}
\prod_{\substack{ 1\leq j \leq  k
\\ p_{j} \nmid t}} \frac{h(p_j)}{p_j}
\ll
 \mu_{\c{P}_i}^{k}
\]
that is proved in~\cite[pg.23]{granvsound}.
\end{proof}

\subsection{Auxiliary facts from probability theory}
\label{s:prelimprobthry}
In this section we  recall  some 
necessary notions from probability theory.

Firstly,
we  need the following notion from~\cite[pg.20]{MR1700749}.
Let $X,Y$ be two metric spaces
and
denote the corresponding $\sigma$-algebras by $\c X$ and $\c Y$.
Assume  that we are given a function $h:X\to Y$ such that if $A\in \c Y$  then $\{x\in X: h(x) \in A\}\in \c X$.
If $\nu$ is a probability measure on $(X,\c X)$ then we can define 
a probability measure on $(Y,\c Y)$
(that is denoted by $\nu h^{-1}$)
as follows:
for any
$A\in \c Y$ we let
\beq{def:inversemeasure}{
(\nu h^{-1})(A):=
\nu(x\in X:h(x) \in A)
.}

We will later need the following result 
from~\cite[Th.29.4]{MR1700749}.
\begin{lemma}[Cr{\'a}mer--Wold]
\label{thm:cramerwold}
For random vectors $$\b X_m=(X_{m,1},\ldots,X_{m,k})
\text{ and }
 \b Y_m=(Y_{m,1},\ldots,Y_{m,k}),$$
a necessary and sufficient condition for 
the convergence in distribution of $\b X_m$ to $\b Y$ is 
that 
\[\sum_{i=1}^ka_i X_{m,i}
\]
converges in distribution to $\sum_{i=1}^k a_i Y_{i}$ for each $\b a\in \R^k$. 
\end{lemma}

Let $\b t\in [0,1]^k$.
Recalling the meaning of $(D, \c D)$ 
in \S\ref{s:statingthmmain}
allows to consider the function $\pi_{\b t}:D\to\R^k$
that is defined through
\[
\pi_{(t_1,\ldots,t_k)}(y) := (y(t_1),\ldots,y(t_k)).
\]

According to~\cite[pg.138]{MR1700749},
if $\b P$ is a probability measure on $(D,\c D)$
then the set 
\[T_{\b P}:=
 \{0,1 \}
\cap
\big\{
t\in (0,1): 
\b P[x\in D:x(t)\neq \lim_{\substack{s\to t\\ s<t}}x(s) ]=0
\big\}
\]
has complement in $[0,1]$ that is countable.
Next, we shall need the definition in~\cite[Eq.(12.27)]{MR1700749}.
Namely, for  a
function $u:[0,1]\to\mathbb{R}$ 
and any $\delta>0$
we define 
\[
w''(\delta,u):=
\sup_{\substack{t_1,t_,t_2 \in [0,1] \\ t_1 \leq  t \leq t_2 \\ t_2-t_1\leq \delta }} 
\min\big\{
|u(t)-u(t_1)|,
|u(t_2)-u(t)|
\big\} .\]
The following result  can be found in~\cite[Th.13.3]{MR1700749}.
\begin{lemma}
\label{thm:billithm}
Suppose that $P $
 and
 $(P_m)_{m\in \N}$ are probability measures on $(D,\c D)$.
If 
\beq{req:billithm1}{
P_m  \pi_{\b t}^{-1}  \text{ converges in probability to } P  \pi_{\b t}^{-1}  
\text{ whenever } \b t \in T_{P}^k ,}
\beq{req:billithm2}{\text{ for every } \epsilon>0 \text{ we have } 
\lim_{\delta\to 0 } P\big [u\in D:|u(1)-u(1-\delta)|\geq \epsilon\big ]=0  }
and for each $\epsilon,\eta>0$
there exists
$\delta \in (0,1), m_0 \in \N$
such that for all $m\geq m_0$ we have 
\beq{req:billithm3}{
P_m\big[u\in D:w''(\delta,u) \geq \epsilon \big] \leq \eta
,}
then
$P_m$ converges in probability to $P$.
\end{lemma}

Recall that $D$ is a metric space whose 
metric is 
given by 
\beq{eq:skorohmetri}{
d(X,Y):=\inf_{\lambda \in \Lambda} 
\max\Big\{
\sup\{   |\lambda(t)-t| :t\in [0,1]  \},
\sup\{   |X(t)-Y(\lambda(t))| :t\in [0,1]  \}
\Big\}} whenever 
$X,Y \in D$ and 
where $\Lambda$ denotes the set of all strictly increasing, continuous maps $\lambda:[0,1]\to[0,1]$,
see, for example~\cite[Eq.(12.13)]{MR1700749}.

To verify~\eqref{req:billithm3} in a specific situation
we shall later need the following two results.

\begin{lemma}
[Theorem 11.3, \cite{MR0466055}]
\label{lem:k622}
Let $P$ be any probability measure on $(D,\c D)$.
Assume that 
$0=s_0<s_1\cdots<s_k=1$
and
$s_i-s_{i-1}\geq \delta$,
$i=1,\cdots,k$,
then 
\[
P\l[u\in D:w''(u,\delta)>\epsilon\r]
\leq 
\sum_{i=0}^{k-2}
P\!\l[u\in D:\epsilon<
\!\!\!
\sup_{\substack{t_1,t,t_2 \in [0,1]^3
\\ 
s_i \leq  t_1 \leq t \leq t_2 \leq  s_{i+2}
\\ }}
\!\!
\min\l\{
|u(t)-u(t_1)| ,
|u(t_2)-u(t)|
\r\}
 \r]
.\]  
 \end{lemma}
The second result
corresponds to the case with 
$\alpha=1=\beta$
of~\cite[Th.10.1]{MR1700749}.
Let $\xi_1,\ldots,\xi_N$ be random variables on a probability space $(\Omega_1,P_1)$
and define 
\[m_{ijk}:=\min\l\{
\l|\sum_{h=i+1}^{j}\xi_h \r|,
\l|\sum_{h=j+1}^{k}\xi_h\r|
\r\},
\
\
0\leq i\leq j \leq k\leq N
.\]\begin{lemma}\label{lem:billithm1}
Suppose that $u_1,\ldots,u_N$ are non-negative numbers with 
\[P_1\l[m_{ijk}\geq \lambda \r]\leq \frac{1}{\lambda^4}
\l(\sum_{i<l\leq k}u_l\r)^2
, \ \ 
0\leq i \leq j \leq k \leq N,\] for $\lambda>0$.
Then, for $\lambda>0$,  
\[
P_1\l[m_{ijk}
\geq \lambda \r]\ll \frac{1}{\lambda^4}
 \l(
\sum_{0<l\leq N}u_l
\r)^2
,\] where the implied constant is 
absolute.
 \end{lemma}

\subsection{Pointwise convergence}
\label{s:pointwiseconv}
Define 
$\psi :\R_{\geq 3} \to \R$
through
\beq{def:bwv911}{ \psi(B):=
(\log \log B)^{-\frac{1}{4}}
  .}
For   $x\in \P^n(\Q)$ and  $B\in \R_{\geq 3}$ we bring into play the function
$
Y_{B}(\bullet ,x)
:[0,1]\to \R$
given by
\begin{equation}\label
{def:pathflat}
t\mapsto
Y_{B}(t,x)
:= 
\frac{1}{(\Delta(f) \log \log B)^{\frac{1}{2}  }} 
\sum_{\substack{ p\leq \exp(\log^t B) \\   
 \log B<p \leq  B^{\psi(B)}   
 }}
\begin{cases} 1-\sigma_p, &\mbox{if } f^{-1}(x)(\Q_p)=\emptyset,\\ 
-\sigma_p, & \mbox{otherwise}.\end{cases}  \end{equation} 
This is a truncated version of the function in~\eqref{eq:behavesas}.
The truncation is introduced for technical reasons.

For $r\in \Z_{\geq 0}$
we denote the $r$-th moment of the standard normal distribution by 
\[
M_r:= 
\begin{cases} \frac{1}{2^{r/2}}\frac{r! 
}{   \l( r/2  
 \r)!}, & r \text{ even, }\\  0, & r \text{ odd.} \end{cases} 
\]
\begin{lemma}\label{lem:wisepoint}
Keep the assumptions of Theorem~\ref{thm:main}.
For every 
$B\geq 3$,
$m\in \N$,
$\b{k} \in \Z_{\geq 0}^m$,
$\b{a}\in \R^m$
and 
$\b t \in [0,1]^m$ with 
 $0\leq t_1   < \ldots < t_m \leq 1 $
we consider the sum 
\[\sum_{\substack{ x \in \PP^n(\QQ), H(x)\leq B\\ f^{-1}(x) \text{ smooth}}}
\prod_{i=1}^m 
\Bigg( 
\sum_{\substack{
\log B<p \leq  B^{\psi(B)} 
\\
 \exp(\log^{t_i} B)<p\leq \exp(\log^{t_{i+1}} B)    
}}
\begin{cases} 1-\sigma_p, &\mbox{if } f^{-1}(x)(\Q_p)=\emptyset,\\ 
-\sigma_p, & \mbox{otherwise}
\end{cases}
\Bigg)^{k_i}
,\] where by convention we set $0^0:=1$.
Letting
$r:=k_1+\cdots+k_m$,
the sum equals
\[
 \# \Omega_B 
\l(\prod_{i=1}^m M_{k_i}  (t_{i+1}-t_{i})^{\frac{k_i}{2}}   \r)
(\Delta(f) \log \log B)^{\frac{r}{2}} 
+
O_{\b a,\b k, \b t,m} \Big( \# \Omega_B (  \log \log B)^{\frac{r-1}{2}} 
\Big) 
.\]
\end{lemma}
\begin{proof}
We shall assume that $t_1=0$ and $t_m=1$,  
an obvious modification of our arguments makes available the proof when $(t_1,t_m)\neq (0,1)$.
Let us  
define the multiset 
 \[\c{A}:=\Big\{a_x := \hspace{-0,3cm} 
\prod_{ 
\substack{p \text{ prime }\\
f^{-1}(x)(\Q_p)=\emptyset}}
\hspace{-0,3cm} p:
\, \,
 x \in \PP^n(\QQ), H(x)\leq B, f^{-1}(x) \text{ smooth}\Big\},\]
the sets of primes 
\[\c P:=\Big\{p \text{ prime}:\log B<p \leq B^{\psi(B)}\Big\}
,\]
\[\c P_i:=\Big\{p \in \c P:
 \exp(\log^{t_i} B)<p\leq \exp(\log^{t_{i+1}} B)
\Big\}
, (1\leq i \leq m)\] 
and introduce
the multiplicative function 
$h:\N \to\R$ as 
$h(Q):=Q\prod_{p\mid Q} \sigma_p$.
In the terminology of~\S\ref{s:intervcorr}
the sum in our lemma takes the shape
\[ 
\sum_{a\in \c{A}}
\prod_{i=1}^m
(\omega_{\c{P}_i}(a)
-\mu_{\c{P}_i}
)^{k_i}
.\]
Recalling~\eqref{def:themusandthesigmas} 
and using Lemma~\ref{lem:abelpu8}
we see that
\[\mu_{\c{P}_i}
=\begin{cases} 
t_2       \Delta(f)  \log \log B+O_{\b a,\b t}(\log \log \log B), &\mbox{if } i=1 ,\\ 
(t_{i+1}-t_{i}) \Delta(f)  \log \log B+O_{\b a,\b t}(1), &\mbox{if }  1<i<m-1,\\ 
(1-t_{m-1}) \Delta(f)  \log \log B+O_{\b a,\b t}(\log \log \log B), &\mbox{if } i=m-1,
\end{cases}
\] 
which can be written as 
\[\mu_{\c{P}_i}=(t_{i+1}-t_{i}) \Delta(f)  \log \log B+O_{\b a,\b t}(\log \log \log B),
\ (1\leq i \leq m-1) .\]
By~\eqref{def:themusandthesigmas}
and Lemma~\ref{lem:harmonicgfts6}
we have 
$\sigma_{\c{P}_i}^2=\mu_{\c{P}_i}+O(1)$, hence 
\[\sigma_{\c{P}_i}=((t_{i+1}-t_{i}) \Delta(f)  \log \log B)^{\frac{1}{2}}
+O_{\b a,\b t}(\log \log \log B),  \ (1\leq i \leq m-1) .\]
This allows to deduce that 
  the product in the right side of~\eqref{eq:timeforlunch1}
equals
\begin{align*}
&\#\big\{x\in \Omega_B: f^{-1}(x) \text{ smooth}\big\}
\prod_{i=1}^m
\Big(M_{k_i} \sigma_{\c{P}_i}^{k_i}
\Big(1+
O_{\b a,\b k, \b t}\Big(\frac{1}{ \log \log B }\Big)
\Big)
\Big)\\
=&
\frac{\#\big\{x\in \Omega_B: f^{-1}(x) \text{ smooth}\big\}}{
(\Delta(f) \log \log B)^{-\frac{r}{2}}}
\!
\l(\prod_{i=1}^m M_{k_i}  (t_{i+1}-t_{i})^{\frac{k_i}{2}}   \r)
\!\! 
\l(\!1\!+\!O_{\b a,\b k, \b t} \Big(\frac{\log \log \log B}{ \log \log B }\Big)\!\r)
.\end{align*}
Similarly,  the product in the right side of~\eqref{eq:timeforlunch2}
is
$\ll_{\b a,\b k, \b t}
B^{n+1}
( \log \log B)^{\frac{r-1}{2}}
$.
Using the estimate $\#\big\{x\in \Omega_B: f^{-1}(x) \text{ smooth}\big\}=\#\Omega_B+O(B^{-1})$
 we can put both formulas in the succinct form
\[
 \# \Omega_B 
\l(\prod_{i=1}^m M_{k_i}  (t_{i+1}-t_{i})^{\frac{k_i}{2}}   \r)
(\Delta(f) \log \log B)^{\frac{r}{2}} 
+
O_{\b a,\b k, \b t} \Big( \# \Omega_B (  \log \log B)^{\frac{r-1}{2}} 
\Big) 
.\]
Therefore, 
Theorem~\ref{prop:epi} shows that the sum in our lemma equals  
\begin{equation}\label{eq:mieng}
 \begin{aligned}
        \# \Omega_B &=\l(\prod_{i=1}^m M_{k_i}  (t_{i+1}-t_{i})^{\frac{k_i}{2}}\r)(\Delta(f) \log \log B)^{\frac{r}{2}} \\
        & + O_{\b a,\b k, \b t} \Big( \# \Omega_B (  \log \log B)^{\frac{r-1}{2}}  + (\log \log B)^m \c{E}_{\c{P}_1,\ldots,\c{P}_m}(\c{A},h,\b{k}) \Big).
       \end{aligned}
  \end{equation} 
It remains to bound the quantity $\c{E}$ above. 
By~\eqref{eq:letusnfridcofbnfgh} 
it is at most 
\[
 \sum_{\substack{
Q\in \N
, \omega(Q)\leq r
\\  p\mid Q \Rightarrow p\in \c{P} 
}} 
 \mu(Q)^2
\big| \c{W}(Q )   \big|
.\] 
Now define the functions
$
t_0(B):=(\log \log B)^\c{C} \ \text{ and } \  t_1(B)=B^{\epsilon_1}
$,
where $\c C:=2r+m$ and $\epsilon_1:=(8r(n+1))^{-1}$.
We certainly have 
$t_0(B)<\log B<B^{\psi(B)} <  t_1(B)$
for all sufficiently large $B$,
thus
the last sum over $Q$ is at most 
\[
\sum_{\substack{
Q\in \N,
  \omega(Q)\leq r
\\  p\mid Q \Rightarrow p\in (t_0(B),t_1(B) ] 
}} 
 \mu(Q)^2
\big| \c{W}(Q )   \big|
.\]
This quantity 
occurs 
also in the proof of~\cite[Prop.3.9]{arXiv:1711.08396},
where it is shown to be 
\[  
\ll_{r,\c{C},\epsilon_1} 
B^{n+1} (\log \log B)^{r-1-\c{C}}
.\] This yields immediately
$
(\log \log B)^m \c{E}_{\c{P}_1,\ldots,\c{P}_m}(\c{A},h,\b{k}) \ll_{\b a,\b k, \b t}
 \# \Omega_B (  \log \log B)^{\frac{r-1}{2}}    
$,
which, in light of~\eqref{eq:mieng}, is sufficient for our proof. 
\end{proof}

\begin{proposition}\label{prop:pointwise}
Keep the assumptions of Theorem~\ref{thm:main}.
Let $\b t \in [0,1]^m$ with 
$$0\leq t_1   < \ldots < t_m \leq 1$$
and assume 
that 
$S_1,\ldots,S_m$ are Lebesgue-measurable subsets of $\R$.
Then 
 \[
\lim_{B\to+\infty}
\b P_B\Big(x\in \Omega_B: 1\leq i \leq m \Rightarrow  Y_{B}(t_i ,x) \in S_i  \Big)
 =
\prod_{\substack{ 1\leq i \leq m \\ t_i \neq 0 }}
 \int_{S_i  }   
\frac{ \exp ( -\theta^2 /2t_i  )   }{(2\pi t_i)^{ \frac{1}{2}   }}   
\mathrm{d}  \theta 
.\]
\end{proposition}
\begin{proof}
We assume that $t_1>0$ but the proof can be easily modified  
when $t_1=0$.
Let us now assume that $Z_0,Z_1,\ldots,Z_m$ 
are random variables on a probability space $(\Omega,P)$
such 
that they are independent in pairs,
that for every $1\leq i \leq m$ the random variable $Z_i$ follows the normal distribution with mean $0$ and variance $t_i$
and  that $Z_0$ assumes the value $0$ with probability  $1$.
Therefore, for any $S_i$ as in the statement of the proposition
we have  
\[
 P(
\b Z\in S_1\times \cdots \times S_m
 )
=
\prod_{\substack{ 1\leq i \leq m   }}
 \int_{S_i  }   
\frac{ \exp ( -\theta^2 /2t_i  )   }{(2\pi t_i)^{ \frac{1}{2}   }}   
\mathrm{d}  \theta 
.\]
By Lemma~\ref{thm:cramerwold}
it is sufficient to show that 
for every $\b a\in \R^m$ the random variable 
\[
\sum_{i=1}^m
a_i 
Y_{B}(t_i,x)
\] defined on $(\Omega_B,\b P_B)$
converges in distribution to  
$
\sum_{1\leq i \leq m} 
a_i Z_i
$ as $B\to+\infty$. 
Let 
 $$\Omega_B^*
:=\{x\in \Omega_B:f^{-1}(x) \text{ smooth}\}$$ and
denote the indicator function of a set $S$ by  $\b 1_S$.
The estimate
$\b P_B( \Omega_B \setminus \Omega_B^*
 ) \ll B^{-1}$
shows that
it suffices to show that \[
\b 1_{\Omega_B^*}(x)
\sum_{i=1}^m
a_i 
Y_{B}(t_i,x)
\] defined on $(\Omega_B,\b P_B)$
converges in distribution to  
$
\sum_{1\leq i \leq m}    a_i Z_i
$.
We will do so by
using the method of moments (see~\cite[Th.30.2]{MR1324786}),
thus,
our proposition
would follow 
from verifying 
\beq{eq:wouldflblut1}{
\frac{1}{\#\Omega_B}
\lim_{B\to+\infty}
\sum_{x\in \Omega_B^*}
\l(
\sum_{i=1}^m
a_i 
Y_{B}(t_i,x)\r)^r
=\int_{\R}
\theta^r
P\l(\sum_{i=1}^m    a_i Z_i  \leq \theta\r)
\mathrm{d}\theta, 
\
(r\in \Z_{\geq 0})
.}
We begin by simplifying
 the right side of~\eqref{eq:wouldflblut1}. 
Whenever
$1\leq i \leq m-1$ we define 
$b_i:= a_i+a_{i+1}+\cdots+a_k$ so that for every $1\leq i \leq m-1$
we can write 
$a_i=b_{i+1}-b_i$.
Thus
\[
\sum_{i=1}^{m}a_i Z_i
=
b_1Z_0+ 
\sum_{i=1}^{m} b_i (Z_{i}-Z_{i-1})
,\]
from which we deduce that $\sum_{i=1}^{m}a_i Z_i$
is a random variable that follows the normal distribution 
with mean $0$ and variance
$\sum_{1\leq i \leq m} b_i^2 (t_{i}-t_{i-1})$.
This immediately yields
\beq{eq:wouldflblut2}{
\int_{\R}
\theta^r
P\!\l(\sum_{i=1}^m    a_i Z_i  \leq \theta\r)
\mathrm{d}\theta=
M_{r}\l(\sum_{i=1}^{m}    b_i^2 (t_{i+1}-t_{i})\r)^{\frac{r}{2}} 
.}
We continue with the treatment of the left side of~\eqref{eq:wouldflblut1}.
 Let  $\c P$ denote the set of all primes in the interval $(\log B,B^{\psi(B)}]$
and set
\[\c P_1:=\c P\cap  (1 ,\exp(\log^{t_{1}} B)],
\c P_i
:=\c P\cap (\exp(\log^{t_{i-1}} B) ,\exp(\log^{t_{i}} B)],
\ (2\leq i \leq m)
.\]
 We see that 
\[\sum_{i=1}^m
 a_i Y_{B}(t_i,x)=
\frac{1}{(\Delta(f) \log \log B)^{\frac{1}{2}  }} 
\sum_{i=1}^{m} b_i 
\sum_{p \in \c P_i}
\begin{cases} 1-\sigma_p, &\mbox{if } f^{-1}(x)(\Q_p)=\emptyset,\\ 
-\sigma_p, & \mbox{otherwise}.\end{cases}
\]
Thus the multinomial theorem
yields
\begin{align*}
\sum_{\substack{ x \in \PP^n(\QQ), H(x)\leq B\\ f^{-1}(x) \text{ smooth}}}
\l(\sum_{i=1}^m a_i Y_{B}(t_i,x)\r)^r
=&
\sum_{\substack{\b k \in \Z_{\geq 0}^{m-1}\\k_1+\cdots+ k_{m}=r }}    
\frac{r! }{k_1!\cdots k_{m}!}
\frac{b_1^{k_1} \cdots b_{m}^{k_{m}} }{(\Delta(f) \log \log B)^{\frac{r}{2}  }}
\times \\ 
\times &
\sum_{\substack{ x \in \PP^n(\QQ), H(x)\leq B\\ f^{-1}(x) \text{ smooth}}}
\prod_{i=1}^{m} \l( 
\sum_{p \in \c P_i}
\begin{cases} 1-\sigma_p, &\mbox{if } f^{-1}(x)(\Q_p)=\emptyset,\\ 
-\sigma_p, & \mbox{otherwise}
\end{cases}
\r)^{k_i}
,\end{align*}
where by convention we set $0^0:=1$.
Invoking
Lemma~\ref{lem:wisepoint}
shows that this is 
 \[\# \Omega_B 
\bigg (\sum_{\substack{\b k \in \Z_{\geq 0}^{{m}}\\k_1+\cdots+ k_{m}=r }}    
\frac{r! b_1^{k_1} \cdots b_{m}^{k_{m}}}{k_1!\cdots k_{m}!}
\prod_{i=1}^{m} 
M_{k_i}  (t_{i+1}-t_{i})^{\frac{k_i}{2}}   \bigg )
 +
O_{\b a, \b t,m,r} \bigg ( \frac{\# \Omega_B}{(  \log \log B)^{\frac{1}{2}}} \bigg )
.\]
Recalling that $M_{k_i}$ vanishes if $k_i$ is odd shows that the sum over $\b k$
zero if $r$ is odd.
If $r$ is even
 we let $r=2s$ and $k_i=2 u_i $ to write the sum over $\b k$ as 
\[\sum_{\substack{\b u \in \Z_{\geq 0}^{m}\\u_1+\cdots+ u_{m}=s }}    
\frac{(2s)! b_1^{2 u_1} \cdots b_{m}^{2u_{m}}}{(2u_1)!\cdots (2u_{m})!}
\prod_{i=1}^{m} \frac{(2u_i)!}{u_i!
2^{u_i} } (t_{i+1}-t_{i})^{u_i}   
=
M_{r}\l(\sum_{i=1}^{m}    b_i^2 (t_{i+1}-t_{i})\r)^{\frac{r}{2}} 
.\]
Using this with~\eqref{eq:wouldflblut2} 
verifies~\eqref{eq:wouldflblut1},
which completes our proof. 
\end{proof}

\subsection{Tightness}
\label{s:tightness} 
Our aim in this section is to prove Proposition~\ref{prop:tightness}, 
which is one of the main ingredients in the proof of Theorem~\ref{thm:main}.

Recall the definition of $\theta_p$ in~\eqref{eq:thesol}.
\begin{lemma}
\label{lem:stackyinequalityfgt6}
Keep the assumptions of Theorem~\ref{thm:main}.  
Then 
for all 
$\b{y} \in  
\R_{\geq 1}^3$
with $ y_1\leq y_2 \leq y_3$
the following bound holds 
with an 
implied constant depending at most on $f$,
\[
 \sum_{\substack{ x \in \PP^n(\QQ), H(x)\leq B\\ f^{-1}(x) \text{ smooth}}}
\
 \prod_{i=1}^2
\left(
 \sum_{\substack{ y_i<p\leq y_{i+1}
\\ \log B<p \leq   B^{\psi(B)} 
}}
 (\theta_{p}(x) -\sigma_{p}) 
\right)^2
\ll B^{n+1}\l(1+ \sum_{\substack{ y_1<p\leq y_3
\\ \log B<p \leq   B^{\psi(B)} 
}
}\frac{1}{p} \r)^2
.\]
\end{lemma}
\begin{proof}
We will make use of Theorem~\ref{prop:epi} with $m=2=k_1=k_2$, 
\[\c P:=\{p \text{ prime} : \log B<p \leq B^{\psi(B)}\},
\c P_i:=\{p \in \c P : y_i<p \leq y_{i+1}    \}, (i=1,2),\]
and with $\c A$, $h(p)$ being   as in the proof of Lemma~\ref{lem:wisepoint}.
According to~\eqref{def:themusandthesigmas}
we have
\[
\sigma_{\c P_i}^2
\leq 
\mu_{\c P_i}
= \sum_{\substack{ y_i<p\leq y_{i+1}
\\\log B<p \leq   B^{\psi(B)} 
}} 
\hspace{-0,2cm}
\sigma_p
, (i=1,2).\]
Therefore, $\sigma_{\c P_i}^2
(1+O(\sigma_{\c P_i}^{-2}
))\ll
1+\mu_{\c P_i}
$ with an absolute implied constant.
Injecting this into~\eqref{eq:timeforlunch1}
we obtain that the sum over $x$ 
in our lemma is 
\[
\ll
 (1+\mu_{\c{P}_1})
 (1+\mu_{\c{P}_2}) 
\left(
B^{n+1}
+
 \c{E}_{\c{P}_1,\c{P}_2}(\c{A},h,(2,2))
\right)
.\]
We can bound the quantity $\c E$ above as in the proof of 
  Lemma~\ref{lem:wisepoint}.
This would give  \[ 
|
\c{E}_{\c{P}_1,\c{P}_2}(\c{A},h,(2,2))
|
\leq 
\sum_{\substack{
Q\in \N,
  \omega(Q)\leq 4
\\  p\mid Q \Rightarrow p\in ((\log \log B)^\c{C}, 
B^{\epsilon_1} ] 
}} 
\hspace{-0,5cm}
 \mu(Q)^2
\big| \c{W}(Q)   \big|
\ll_{r,\c{C},\epsilon_1} 
B^{n+1} (\log \log B)^{3-\c{C}}
,\] so that, taking $\c C=3$,
the sum over $x$
in our lemma
becomes
\[\ll
B^{n+1}
 (1+\mu_{\c{P}_1})
 (1+\mu_{\c{P}_2}) 
.\]
Using Lemma~\ref{lem:harmonicgfts6}
shows that 
\[ 
\mu_{\c{P}_i}
=
\sum_{\substack{p\in (y_i,y_{i+1}]
\\  \log B < p \leq B^{\psi(B)}
}}
\sigma_p
\ll 
\sum_{\substack{p\in (y_i,y_{i+1}]
}} \frac{1}{p}
,\] where the implied constant depends only on $f$.
Thus the sum over $x$ in the lemma is 
\[
\ll B^{n+1}
\prod_{i=1}^2
\l(1+\sum_{y_i<p\leq y_{i+1}}\frac{1}{p}\r)
.\]
Using the inequality $(1+\epsilon_1)(1+\epsilon_2)
\leq (1+\epsilon_1+\epsilon_2)^2$, valid whenever both $\epsilon_i$ are non-negative, 
concludes the proof. 
 \end{proof}

Define for $y_1,y_2,y_3 \in [0,1]$ with $y_1\leq y_2\leq y_3$, 
$B\geq 3$
and $x\in \P^n(\Q)$ the function 
\[
\Psi_\b{y}(x,B)
:=
\min
\Bigg\{
\Bigg|\sum_{\substack{ y_i<p\leq y_{i+1}
\\\log B<p \leq   B^{\psi(B)} 
}}
\frac{(\theta_{p}(x) -\sigma_{p})}{(\Delta(f) \log \log B)^{1/2}}
\Bigg|:
i=1,2
\Bigg\}
.\] 
\begin{lemma}
\label{lem:pre1723}
Keep the assumptions of Theorem~\ref{thm:main}. Then 
for all 
$\lambda>0$ and 
$\b{y} \in  
\R_{\geq 1}^3$
with $ y_1\leq y_2 \leq y_3$
the following  holds 
with an 
implied constant depending at most on $f$, 
\[ 
\b P_B[
x\in \Omega_B:
\Psi_\b{y}(x,B) \geq \lambda ] 
\ll
\frac{1}{(\log \log B)^2
\lambda^4}
\l(1+
 \sum_{\substack{ 
y_1<p\leq y_3
\\
\log B<p \leq B^{\psi(B)}
}} \frac{1}{p}
\r)^2
.\] 
\end{lemma}
\begin{proof}
The bound $\b P_B(x\in \Omega_B
:f^{-1}(x) \text{ singular})
\ll B^{-1}$ 
 shows that 
\begin{equation}\label{eq:daftb}
\b P_B
[
x\in \Omega_B:
\Psi_\b{y}(x,B) \geq \lambda ] 
=
\b P_B[
x\in \Omega_B
:f^{-1}(x) \text{ smooth},
\Psi_\b{y}(x,B) \geq \lambda ] 
+O\Big(\frac{1}{B}\Big)
.\end{equation}
Note that if $\Psi_\b{y}(x,B) \geq \lambda $ 
then 
\[\lambda^2
\leq 
\prod_{i=1}^2
\left| 
\sum_{\substack{ y_i<p\leq y_{i+1}
\\\log B<p \leq   B^{\psi(B)} 
}}
\frac{(\theta_{p}(x) -\sigma_{p})}{(\Delta(f) \log \log B)^{1/2}}
\right| 
.\]
Thus, the entity $\b P_B[.]$ on the 
right side of~\eqref{eq:daftb} 
is bounded by the following quantity due to
Chebychev's inequality,
\[
\frac{1}{\lambda^4
\#\Omega_B
}
\sum_{\substack{ x \in \PP^n(\QQ), H(x)\leq B\\ f^{-1}(x) \text{ smooth}}}
\
\prod_{i=1}^2
 \left(\sum_{\substack{ y_i<p\leq y_{i+1}
\\\log B<p \leq   B^{\psi(B)} 
}}
\frac{(\theta_{p}(x) -\sigma_{p})}{(\Delta(f) \log \log B)^{1/2}}
 \right)^2
.\] Alluding to
Lemma~\ref{lem:stackyinequalityfgt6}
concludes the proof.\end{proof}

Recall~\eqref{def:pathflat}.
Define for
$\lambda
>0,B\geq 3$
and
 $s,s'\in [0,1]$ with $s\leq s'$,
\[
\Gamma_{\lambda
,B}(s,s')
\!:=\!
\b P_B
\Bigg(
x\in \Omega_B:\!
\lambda
<
\!\!\!\!
\sup_{\substack{ t_1,t,t_2 \in [0,1] \\ s\leq t_1\leq t \leq t_2 \leq s'  }}
\min\Big\{|Y_B(t,x) -Y_B(t_1,x)  | ,
|Y_B(t_2,x) -Y_B(t,x)  |  
\Big\}
\Bigg)
.\]
\begin{lemma}
\label{lem:1723}
Keep the assumptions of Theorem~\ref{thm:main}. 
For all 
$\lambda>0$ and any 
$s,s' \in [0,1]$ with 
 $s<s'$  
there exists $B_0$ 
that depends at most on $f$ and $s'-s$
such that 
if 
$B\geq B_0$ then 
\[\Gamma_{\lambda,B}(s,s') 
\ll 
\frac{(s'-s)^2}{\lambda^4}
\] with an implied constant that depends at most on $f$.
\end{lemma}
\begin{proof}
Order all
primes in  
$\{p:
\mathrm{e}^{\log^{s} B }< p\leq \mathrm{e}^{\log^{s'} B } ,
\log B<p \leq  
B^{\psi(B)} 
\}
$
as $p_1<\ldots<p_N$, with the convention that $N=0$ if the set is empty.
For every $1\leq i \leq N$ 
we define the random variable
 $\xi_i$
on
the probability space $(\Omega_B,\b P_B)$ through
\[
\xi_i(x):=
\frac{(\theta_{p_i}(x) -\sigma_{p_i})}{(\Delta(f) \log \log B)^{1/2}}
, x\in \Omega_B.\] 
For any $i,j,k$
with $0\leq i \leq j \leq k \leq N$,
any $B\geq 3$ and 
 $x\in \Omega_B$
let us bring into play 
\[
m_{ijk}(x)
:=\min\Bigg\{
\frac{|\sum_{h=i+1}^{j}(\theta_{p_h}(x) -\sigma_{p_h})|}{(\Delta(f) \log \log B)^{1/2}},
\frac{|\sum_{h=j+1}^{k}(\theta_{p_h}(x) -\sigma_{p_h})|}{(\Delta(f) \log \log B)^{1/2}}
\Bigg\}
.\] In particular, one has 
\[
\Gamma_{\lambda,B}(s,s')
=\b P_B\left(x\in \Omega_B:
\lambda<
\max_{0\leq i \leq j \leq k \leq N}
m_{ijk}(x)
\right)
.\] 
Note that Lemma~\ref{lem:pre1723}
allows to apply 
Lemma~\ref{lem:billithm1}
with $P_1=\b P_B$ and  
$u_l=1/p_l$.
Thus 
\[\b P_B\Big[x\in \Omega_B:
\lambda  
\leq 
\max_{0\leq i \leq j \leq k \leq N}
m_{ijk}(x)
\Big] \ll  
\frac{  
1 
}{(\log \log B)^2\lambda^4}
\Bigg(\sum_{\substack{
\log B<p \leq   B^{\psi(B)} 
\\
\mathrm{e}^{\log^{s} B }< p\leq \mathrm{e}^{\log^{s'} B } 
}} \frac{1}{p}
\Bigg)^2
.\]
Ignoring the condition $\log B<p \leq   B^{\psi(B)} $  
and using Lemma~\ref{lem:abelpu9}
 shows that  the sum over $p$
is at most 
$(s'-s ) (\log \log B)+C'$ for some absolute constant $C'>0$.
Taking any  
$B_0$ with  
$C'\leq (s'-s ) (\log \log B_0)$
concludes our proof. 
\end{proof}

\begin{proposition}
\label{prop:tightness}
Keep the assumptions of Theorem~\ref{thm:main}.
There exists  
$K>0$ that depends at most on $f$
such that 
for every $\lambda>0$ and $0<\delta<1$
 there exists $B_0=B_0(f,\delta,\lambda)>0$ 
with 
\[
B\geq B_0
\Rightarrow 
\b{P}_B
\big[x\in \Omega_B:
w''(\delta,   Y_{B}(\bullet ,x)   )\geq \lambda
\big]
\leq \frac{K \delta
}{\lambda^4}
.\]
\end{proposition}  
\begin{proof}
 Let $k=k(\delta)$ be the largest positive integer satisfying
 $\delta k< 1$.
Define  $h:\Omega_B\to D$ through  $h(x):=Y_{B}(\bullet ,x)$ and in the terminology
of~\eqref{def:inversemeasure} define 
$P_2:=\b P_B h^{-1}$.
We use Lemma~\ref{lem:k622} with   $P=P_2$
 and 
\[
s_i:=
\begin{cases}
i\delta,  &\text{ if } i=0,1,\ldots,k-1,\\
1, &\text{ if } i=k.
\end{cases}
\] 
We obtain that the quantity $\b P_B$ in our proposition is at most 
\[ 
\sum_{i=0}^{k-2}
\Gamma_{\epsilon,B}(s_i,s_{i+2}) 
.\]
Using Lemma~\ref{lem:1723} 
we obtain $B_0$ that depends at most on $\delta$ 
such that if $B\geq B_0$ then the sum over $i$ 
is at most 
$
K'
\lambda^{-4}
\sum_{i=0}^{k-2}(s_{i+2}-s_i)^2
$.
For every $i\neq k-2$ we have $s_{i+2}-s_i=2\delta^{-1}$.
Note that by the definition of $k$ we have $(k+1)\delta \geq 1$,
therefore $s_{k-2}=(k+1)\delta-3\delta \geq 1-3\delta$. We obtain $s_k-s_{k-2}\leq 3\delta $.
This gives 
$$\sum_{i=0}^{k-2}
(s_{i+2}-s_i)^2
\leq (k-2)\delta^2+9\delta^2
\leq
9k \delta^2 <9 \delta
 ,$$ which concludes the proof.
\end{proof}

\subsection{Proof of Theorem~\ref{thm:main}}
\label{s:proofthm} 
We modify
 the argument
behind
the analogous statement for completely
additive functions defined on the integers, 
see the work of Billingsley~\cite[Th.4.1]{MR0466055}.
Technical difficulties arise 
owing to the comments in Remark~\ref{rem:levofbaddiszero}.
While our level of distribution is $0$,
the level of distribution in Billingsley's proof 
is at a sharp contrast, namely, it 
attains its maximum value, $1$. To see this, note that 
 the related estimate in his proof is $$\#\{m\in \N\cap [1,n]:m\equiv 0\md{Q}\}=\frac{n}{Q}+O(1)$$
and clearly 
the error term is dominated by the main as long as $Q\leq n^{1-\epsilon}$, where $\epsilon>0$ is arbitrary.

We begin by 
 estimating  the approximation 
of $X_{B}(\bullet ,x)
$ by 
$Y_{B}(\bullet ,x)$
. Recall the definition of the Skorohod metric
in~\eqref{eq:skorohmetri}
and the function $Y_{B}(\bullet ,x)$
in~\eqref{def:pathflat}.
\begin{lemma} \label{lem:tranghfa0}
For every $\epsilon>0$
we have 
\[\b P_B\Big(x\in \Omega_B:
d(X_{B}(\bullet ,x), Y_{B}(\bullet ,x))
 \geq \epsilon
\Big)
\ll_\epsilon (\log \log B)^{-\frac{1}{4}}
.\]
\end{lemma} 
\begin{proof}
Let  $    m(B,t)    :=\min\big\{\exp(\log^t B) ,  B^{\psi(B)}  \big\}$ and 
\begin{align*}
&M_1(B,t):=\sum_{ p\leq  \log  B }
\begin{cases} 1, &\mbox{if } f^{-1}(x)(\Q_p)=\emptyset,\\ 
0, & \mbox{otherwise},\end{cases} 
\\
&M_2(B,t,x) :=\sum_{m(B,t)<p\leq \exp(\log^t B)}
\begin{cases} 1, &\mbox{if } f^{-1}(x)(\Q_p)=\emptyset,\\ 
0, & \mbox{otherwise},\end{cases} 
\\
& M_3(B,t) :=-t \Delta(f) \log \log B+\sum_{\log B<p \leq   m(B,t) }   \sigma_p ,
 \end{align*}  
where empty sums are set equal to zero.
A moment's thought allows one to see that 
\beq{eq:contexty0}{\l(X_{B}(t,x)-Y_{B}(t,x)\r) 
(\Delta(f) \log \log B)^{\frac{1}{2}  }
=M_1(B,t)+M_2(B,t,x) +M_3(B,t).
}
According to Lemma~\ref{lem:fshgtt672},
if $g(x)\neq 0$ then 
\beq{eq:contexty1}{|M_1(B,t)|  \leq 
\sum_{\substack{   p\mid g(x), p\leq \log B }}1 .}
Similarly, if $g(x)\neq 0$ and $H(x)\leq B$ then 
Lemma~\ref{lem:fshgtt672}
ensures that   \[|M_2(B,t,x) |\leq \sum_{\substack{ m(B,t)<p\leq \exp(\log^t B)\\   \ p\mid g(x)}}1
.\]
If $m(B,t)= \exp(\log^t B)$ then this sum is empty and if $m(B,t)=  B^{\psi(B)}$ then 
\beq{eq:contexty2}{|M_2(B,t,x) | \leq 
\sum_{\substack{  p>B^{\psi(B)}  \\     p\mid g(x)      }}1 
\ll \frac{\log |g(x)|}{ \log  (B^{\psi(B)})}
\ll \frac{\log (B^{\deg(g)})}{ \log  (B^{\psi(B)})}
\ll
{(\log \log B)^{\frac{1}{4}}}
}
because a non-zero integer $m$ can have at most $\frac{\log |m|}{\log M}$ prime divisors in the range $p>M$.
To bound $M_3(B,t)$ when  $m(B,t)= \exp(\log^t B)$
we invoke Lemma~\ref{lem:abelpu8} to obtain
\[|M_3(B,t)| \leq 
\Big|-t \Delta(f) \log \log B+\sum_{ p \leq   \exp(\log^t B)}   \sigma_p \Big|+\Big|\sum_{p\leq \log B}\sigma_p\Big|
\ll 1+ \log \log \log B.\]
 In the remaining case $m(B,t)= B^{\psi(B)}$ 
we note  that 
$\leq B^{\psi(B)} \leq \exp(\log^t B)  \leq B$ implies \[
-\log \psi(B)+ \log \log B\leq t \log \log B\leq \log \log B
\]
and therefore $t \log \log B=\log \log B+O(\log \log \log B)$ with an absolute implied constant.
Thus, Lemma~\ref{lem:abelpu8} shows that $M_3(B,t)$ equals 
\[-\log \log B+O(\log \log \log B)+\sum_{\log B<p \leq   B^{\psi(B)} }   \sigma_p 
\ll \log \log \log B.\] 
This shows that for all $x\in \Omega_B$ with $g(x)\neq 0$ one has 
\beq{eq:contexty3}{|M_3(B,t)| \ll   \log \log \log B .} 
Injecting~\eqref{eq:contexty1},\eqref{eq:contexty2}
and~\eqref{eq:contexty3} into~\eqref{eq:contexty0}
shows that if  $H(x)\leq B$ and $g(x)\neq 0$ then 
\beq{def:gigue2}{
\big|X_{B}(t,x)-Y_{B}(t,x)\big|
 \ll
{(\log \log B)^{-\frac{1}{2}}}
\Big(
{(\log \log B)^{\frac{1}{4}}}
+\sum_{\substack{   p\mid g(x),  p\leq \log B }}1
\Big)
,}
where the implied constant is independent of $t$ and $B$.
We may now take $\lambda(t):=t$ in~\eqref{eq:skorohmetri}
to see that $d(X,Y)  \leq \sup\{|X(t)-Y(t)|:t\in [0,1]\}$, therefore 
\beq{eq:contexty4}{
d(X_{B}(\bullet ,x), Y_{B}(\bullet ,x))
\ll
{(\log \log B)^{-\frac{1}{2}}}
\Big(
{(\log \log B)^{\frac{1}{4}}}
+\sum_{\substack{   p\mid g(x),  p\leq \log B }}1
\Big)
.}
Note that since $g$ is not identically vanishing 
we have $\b P_B(x\in \Omega_B:g(x)=0)\ll B^{-1}$.
This shows that the quantity $\b P_B$ in the statement of our lemma equals 
\[
O(B^{-1})+ \b P_B\l(x\in \Omega_B: g(x) \neq 0,
d(X_{B}(\bullet ,x), Y_{B}(\bullet ,x))
 \geq \epsilon \r)
 \] and by 
 Markov's inequality this is
\[\ll
B^{-1}+\frac{1}{\epsilon  B^{n+1} }  
\sum_{\substack{x\in \P^n(\Q), g(x)\neq 0  \\H(x)\leq B  }}
d(X_{B}(\bullet ,x), Y_{B}(\bullet ,x))
.\] Using~\eqref{eq:contexty4} and~\cite[Lem.3.10]{arXiv:1711.08396} 
for 
$z(B)=(\log \log B)^{\frac{1}{4}}, y(B):=\log B$ yields the bound 
\[\ll_\epsilon B^{-1}+
{(\log \log B)^{-\frac{1}{2}}}
\l((\log \log B)^{\frac{1}{4}}+\log \log \log B\r)
,\] which concludes our proof.
\end{proof} 

By~\cite[Th.3.1]{MR1700749} and Lemma~\ref{lem:tranghfa0} 
we see that Theorem~\ref{thm:main}  holds
as long as we prove it with $Y_{B}$ in place of $X_{B}$.
We shall do so by using Lemma~\ref{thm:billithm} 
with $P$ being the Wiener measure $W$ and
$P_B:=\b P_B Y_B^{-1}$. The latter measure is defined on $(D,\c D)$ 
via~\eqref{def:inversemeasure}  with 
$$(X,\c X)=(\Omega_B,\{A:A\subset \Omega_B\} ),
Y:=(D,\c D),
\nu:=\b P_B
$$ and 
$h:(\Omega_B,\b P_B)\to (D,\c D)$ being
given by $x\mapsto Y_B(\bullet,x)$.
In particular, for every $B\geq 1$ and every 
$\delta, \epsilon>0$ we can write 
\beq{eq:reexprvioln}{
P_B\big[u\in D:w''(\delta,u) \geq \epsilon \big] 
=
\b P_B\big[x\in \Omega_B: w''(\delta,Y_B(\bullet,x) \geq \epsilon \big] 
.}
Now fix any $\b t\in [0,1]^m$.
To rephrase~\eqref{req:billithm2}
we use~\eqref{def:inversemeasure}
with  
\[(X,\c X):=(D,\c D),
(Y, \c Y):=(\R,\c B(\R)),
\nu:=P_B
\] 
and $h:D\to \R^m$ defined by 
$u \mapsto \pi_{\b t}(u)$.
Here, $\c B(\R)$ is the standard
Borel $\sigma$-algebra in the real line.
This shows that $P_B  \pi_{\b t}^{-1}$
is a measure on $(\R^m,\c B(\R^m))$
and, in particular, if $S_1\times \cdots \times S_m
 \in \c B(\R)^m$ then 
\[
P_B  \pi_{\b t}^{-1}(S_1\times \cdots \times S_m)
=P_B(u\in D:
1\leq i \leq m \Rightarrow
u(t_i) \in S_i
 ),\] which, as explained above, equals 
$\b P_B(x\in \Omega_B:1\leq i \leq m \Rightarrow Y_B(t_i,x) \in S_i )$. 
A similar construction with $\nu$ replaced by $W$ shows that 
\[P  \pi_{\b t}^{-1}(S_1\times \cdots \times S_m)
=W(u\in D:
1\leq i \leq m \Rightarrow
u(t_i) \in S_i
)
.\]
Recall that part of the definition of the Wiener measure
is that this equals 
\[
\prod_{\substack{ 1\leq i \leq m \\ t_i \neq 0 }}
 \int_{S_i  }   
\frac{ \exp ( -\theta^2 /2t_i  )   }{(2\pi t_i)^{ \frac{1}{2}   }}   
\mathrm{d}  \theta 
.\] This can be seen by taking $(s,t)=(0,t_i)$ in~\cite[Eq.(37.4)]{MR1324786}. 
Therefore,
in our setting,~\eqref{req:billithm1}
is equivalent to
Proposition~\ref{prop:pointwise}.

Let us now see why~\eqref{req:billithm2} is automatically satisfied when $P$ is the Wiener measure.
Alluding to~\cite[Eq.(8.4)]{MR1700749}
we have for $\epsilon,\delta>0$ that
\begin{align*}
W(u\in D:|u(1)-u(1-\delta)| \geq \epsilon )
=&
\frac{1}{\sqrt{2\pi \delta}}
\int_{\R\setminus (-\epsilon,\epsilon)}
\exp(-\theta^2/(2\delta))
\mathrm{d}\theta
\\
=&
\frac{1}{\sqrt{2\pi  }}
\int_{\R\setminus (- \epsilon/\sqrt{\delta},\epsilon/\sqrt{\delta} )}
\exp(-\theta'^2/2)
\mathrm{d}\theta'
.\end{align*}
For fixed $\epsilon$ and for $\delta\to 0$ the last expression is the tail of a convergent integral,
thus it converges to zero.

To complete the proof of 
Theorem~\ref{thm:main}
via
Lemma~\ref{thm:billithm}
it remains to verify~\eqref{req:billithm3}.
Owing to~\eqref{eq:reexprvioln}
we see that~\eqref{req:billithm3}
can be reformulated equivalently as follows: 
for each $\epsilon,\eta>0$
there exists
$\delta \in (0,1), B_0 \in \N$
such that for all $B\geq B_0$ we have 
\[\b P_B\big[x\in \Omega_B: w''(\delta,Y_B(\bullet,x) )
\geq \epsilon \big]  
\leq \eta
.\]
The fact that this holds is verified in Proposition~\ref{prop:tightness}.
This completes the proof of Theorem~\ref{thm:main}.
\qed

\section{Consequences of the Brownian model}
\label{s:feynman_kac}
In this section we give some number theoretic
consequences
of the fact that $p$-adic solubility can be modelled by Brownian motion.

\subsection{Proof of Theorem~\ref{thm:beyondeq} and Corollary~\ref{cor:pjcor}}
\label{s:proofcoro}
Define
 \[A:=\{u\in D:z\leq \max_{0\leq t \leq 1}
u(t)\}
.\]
Owing to the 
 reflection principle
this set has 
  Wiener measure given by 
\[
W(A)=\frac{2}{\sqrt{2\pi}}
\int_z^{+\infty}
\mathrm{e}^{-\frac{t^2}{\!2}}
\mathrm{d}t
,\]
see~\cite[Th.2.21]{MR2604525}.
An application of Theorem~\ref{thm:main}
concludes the proof. \qed
\subsection{Proof of Theorem~\ref{thm:beyondeqpartita}  and Corollary~\ref{cor:pjcorpartita}}
\label{s:needmus}
The set  \[A:=\{u\in D:z\leq 
\max_{0\leq t \leq 1}
|u(t)|
\}
\]
has Wiener measure $W(A)=1-\tau_\infty(z)$ owing to
Donsker's theorem and~\cite[II,pg.292]{MR0015705}.
An application of Theorem~\ref{thm:main}
concludes the proof of
Theorem~\ref{thm:beyondeqpartita}.
To prove Corollary~\ref{cor:pjcorpartita}
we only have to show that 
$\tau_\infty(z)=1+O(|z|^{-2/3}) $.
For $M:=1+|z|^{2/3}$ we see that the series in~\eqref{def:john_lord}
is alternating, thus its tail is bounded by 
\[
\sum_{m > M}   \frac{(-1)^m}{2m+1}
\exp\Bigg\{-\frac{(2m+1)^2\pi^2}{8z^2}\Bigg\}
\ll  \frac{1}{2M+1}
\exp\Bigg\{-\frac{(2M+1)^2\pi^2}{8z^2}\Bigg\}
\ll \frac{1}{M}
.\] By the Taylor expansion $\exp(y)=1+O(y)$, valid when $|y|\ll 1$, we get 
\[
\sum_{0\leq m\leq M}   \frac{(-1)^m}{2m+1}
\exp\Bigg\{-\frac{(2m+1)^2\pi^2}{8z^2}\Bigg\}
=
\sum_{0\leq m\leq M}   \frac{(-1)^m}{2m+1}
+O\l(\frac{M^2}{z^2}\r)
,\] owing to $\sum_{m\leq M}m\ll M^2$. The last sum over $m$ can be completed by introducing an error term of size $\ll 1/M$,
thus giving 
\[
\sum_{0\leq m\leq M}   \frac{(-1)^m}{2m+1}
\exp\Bigg\{-\frac{(2m+1)^2\pi^2}{8z^2}\Bigg\}
=
\frac{\pi}{4}
+O\l(\frac{1}{M}+\frac{M^2}{z^2}\r)
.\]
This completes the proof.
\qed
\subsection{A variant of the path $X_B$}
For a 
prime $p\geq 3$ define 
$p_-$ to be the greatest prime strictly
smaller 
 than $p$
and let $2_-:=1$.
Recall the definition of $\theta_p(x)$
in~\eqref{eq:thesol}.
Before proceeding to the proof of the 
rest of our results 
it is necessary to approximate the path 
$X_{B}(\bullet ,x)$
in Definition~\ref{def:pathsdef} by the following variant:
for each  $x\in \P^n(\Q)$ and  $B\in \R_{\geq 3}$ we define the function $
Z_{B}(\bullet ,x)
:[0,1]\to\R$ as follows, 
\[ 
t\mapsto
Z_{B}(t,x)
:= 
\frac{1}{(\Delta(f) \log \log B)^{\frac{1}{2}  }} 
\Osum_{\substack{ p \leq B   
  }} (\theta_p(x)-\sigma_p),   \] 
where the sum $\Osum$ is taken over all primes $p$ satisfying 
\[
\sum_{q\leq p_-    } \sigma_q\leq \Delta(f) t \log \log B
.\] 
Therefore, labelling all primes 
in ascending order as $q_1=2,q_2=3,\ldots$,
and letting 
\[
T_i(B,x)
:=
\l\{
t:
\sum_{q \text{ prime} \leq q_i} \sigma_q
<\Delta(f) t \log \log B
\leq 
\sum_{q \text{ prime} \leq q_{i+1} } \sigma_q
\r\}
,\]
we infer
\beq{eq:meas}{
\mathrm{meas}
\l(
T_i(B,x)
\r)
=
\frac{\sigma_{q_{i+1}}}{\Delta(f) \log \log B}
}
and  
\beq{eq:stable}
{t\in T_i(B,x) 
\Rightarrow Z_{B}(t,x)=\frac
{\omega_f(x,q_{i+1})  -    \sum_{q\leq q_{i+1}  }   \sigma_q}
{\l(    \Delta(f)   \log \log B\r)^{\frac{1}{2}} }
.}

Recall the definition of  
$\psi(B)$
in~\eqref{def:bwv911}.
\begin{lemma}
\label{lem:English_Suite_No_2}
For   $x\in \P^n(\Q)$ and  $B\in \R_{\geq 3}$ we define 
$
Z'_{B}(\bullet ,x)
:[0,1]\to \R$
given by
\[ 
t\mapsto
Z'_{B}(t,x)
:= 
\frac{1}{(\Delta(f) \log \log B)^{\frac{1}{2}  }} 
\Osum_{\substack{    
\log B<p \leq  B^{\psi(B)}   
   }}
(\theta_p(x)-\sigma_p)
,\] 
where the sum $\Osum$ is taken over all primes $p$ satisfying 
\[
\sum_{q\leq p_-    } \sigma_q\leq \Delta(f) t \log \log B
.\]  
Then for 
every $\epsilon>0$
we have 
\[\b P_B\Big(x\in \Omega_B:
d(Z_{B}(\bullet ,x), Z'_{B}(\bullet ,x))
 \geq \epsilon
\Big)
\ll_\epsilon (\log \log B)^{-\frac{1}{4}}
.\]
\end{lemma}
\begin{proof} 
Ignoring the condition in $\Osum$ gives
\[
| Z_{B}(t,x) - Z'_{B}(t,x)  |
( \Delta(f)
\log \log B)^{\frac{1}{2}  }
\leq
\sum_{p\leq \log B}
(\theta_p(x)+\sigma_p)
+
\sum_{B^{\psi(B)}<p\leq   B}
(\theta_p(x)+\sigma_p)
.\] By  Lemma~\ref{lem:abelpu8}
the $\sigma_p$ terms contribute 
\[\ll \log \log \log B+ \log \frac{\log B}{\log B^{\psi(B)}}
\ll
\log \log \log B
.\]
As in the proof of Lemma~\ref{lem:tranghfa0}, if $g(x)\neq 0$ 
and $H(x)\leq B$ 
then 
the remaining  terms  are
\[
\ll
\sum_{\substack{p\mid g(x)  \\ B^{\psi(B)}< p\leq   B }}   1
+ \sum_{\substack{p\mid g(x)  \\ p\leq \log B }}1
\ll
  \frac{\log |g(x)| }{\log B^{\psi(B)} }
+ \sum_{\substack{p\mid g(x)  \\ p\leq \log B }}1
\ll
\frac{1}{\psi(B)}
+
\sum_{\substack{p\mid g(x)  \\ p\leq \log B }}1,\] hence by~\eqref{def:bwv911} we obtain
\[
| Z_{B}(t,x) - Z'_{B}(t,x)  | 
\ll
{(\log \log B)^{-\frac{1}{2}}}
\Big(
{(\log \log B)^{\frac{1}{4}}}
+\sum_{\substack{   p\mid g(x),  p\leq \log B }}1
\Big)
.\]   
The right side coincides with that in~\eqref{def:gigue2},
and the rest of the proof can now be completed as 
in the proof of Lemma~\ref{lem:tranghfa0}.    
\end{proof}
Recall the definition of $Y_{B}(\bullet ,x)$ 
in~\eqref{def:bwv911}. 
\begin{lemma}
\label{lem:Contrapunctus_XIV}
For every $\epsilon>0$
we have 
\[\b P_B\Big(x\in \Omega_B:
d(Y_{B}(\bullet ,x), Z'_{B}(\bullet ,x))
 \geq \epsilon
\Big)
\ll_\epsilon (\log \log B)^{-\frac{1}{2}}
.\]
\end{lemma}
\begin{proof} 
Let 
$
S_1:=\l\{p\leq \exp\l(
(\log B)^t
\r)\r\}$ and 
\[S_2:=\l\{p\leq B:
\sum_{q\leq p_-} \sigma_p \leq \Delta(f) t \log \log p
\r\}
.\] We infer that 
\[
| Y_{B}(t,x) - Z'_{B}(t,x)  |
( \Delta(f)
\log \log B)^{\frac{1}{2}  }
\leq
\sum_{\substack{
p\in S_2\setminus S_1 
\\
\log B<p \leq  B^{\psi(B)}  
}} (\theta_p+\sigma_p)
+
\sum_{\substack{
p\in S_1\setminus S_2 
\\
\log B<p \leq  B^{\psi(B)}  
}} (\theta_p+\sigma_p)
.\] We will deal with the sum over $p\in S_2\setminus S_1 $
since the other sum can be treated similarly.
For a prime $p$
not in $S_1$ we have 
\[\sum_{q\leq p_-} 
\sigma_q \leq \Delta(f) t \log \log B
,\] hence by $\sigma_p\leq 1$ we have \[
\sum_{q\leq p} 
\sigma_q\leq 
1+ \Delta(f) t \log \log B
.\]
By Lemma~\ref{lem:abelpu8}
there exists a constant $C_1=C_1(f)$ such that 
\[
(\Delta(f)  \log \log p)
-C_1
\leq 
1+ \Delta(f) t \log \log B
,\] hence 
$\log \log p \leq C_2 +t \log \log B
$ for some $C_2=C_2(f)$.
Let us now define $z_1$ and $z_2$ through 
\[
\log \log z_1= t \log \log B
\text{ and }
\log \log z_2= C_2 +t \log \log B
\] and observe that 
if $p\in S_2\setminus S_1$ 
then $z_1<p\leq z_2$.
By Lemmas~\ref{lem:harmonicgfts6} and~\ref{lem:abelpu9}
we have 
\beq{eq:Contrapunctus X}{
\sum_{\substack{
p\in S_2\setminus S_1 
\\
\log B<p \leq  B^{\psi(B)}  
}}  \sigma_p 
\ll \sum_{z_1<p\leq z_2} \frac{1}{p}
=O(1)+(\log \log z_2)-(\log \log z_1)
\ll 1
,}
with an implied constant depending at most on $f$.
We furthermore have  
\[
\sum_{\substack{
p\in S_2\setminus S_1 
\\
\log B<p \leq  B^{\psi(B)}  
}} 
\theta_p(x)
\leq 
\sum_{\substack{
z_2<p\leq z_1
\\ \log B<
p \leq  B^{\psi(B)}
\\  
}} \theta_p(x)
,\]
 hence 
\[
\sum_{\substack{
x\in \P^n(\Q),
H(x) \leq B
\\
f^{-1}(x) \text{ non-singular}
}} 
\sum_{\substack{
p\in S_2\setminus S_1 
\\
\log B<p \leq  B^{\psi(B)}  
}} 
\theta_p(x)
\leq 
\sum_{\substack{
z_2<p\leq z_1
\\ \log B<
p \leq  B^{\psi(B)}
\\  
}} 
\c A_p
,\]
where $\c A_p$
is as in~\eqref{def:econtn}.
Using
Lemma~\ref{lem:harmonicgfts7865} 
and the bound $\c A_1\ll B^{n+1}$ shows that 
 this is 
\[\ll
\sum_{\substack{
z_2<p\leq z_1
\\  \log B<
p \leq  B^{\psi(B)}
\\  
}} \l(
B^{n+1}
\sigma_p
+ \frac{B^{n+1}  }{ p \log B}
+
p^{2n+1}B+ p B^n (\log B)^{[1/n]}
\r)
.\] Invoking~\eqref{eq:Contrapunctus X}
the first term is 
$ \ll B^{n+1}$.
The second term 
is 
\[\ll 
\frac{B^{n+1}}{\log B}
\sum_{p\leq B} \frac{1}{p} \ll B^{n+1}.\]
The third term is 
\[\ll B
\sum_{p\leq B^{\psi(B)}} p^{2n+1}
\ll_\epsilon B^{1+\epsilon}
,\] valid for all $\epsilon>0$.
The fourth term can be bounded by 
\[
\ll B^n (\log B)^{[1/n]}
\sum_{p\leq B^{\psi(B)}}   p  \ll_{\epsilon}
 B^{n+\epsilon}
.\]
We have thus shown that 
\[
\sum_{\substack{
x\in \P^n(\Q),
H(x) \leq B
\\
f^{-1}(x) \text{ non-singular}
}} 
\sum_{\substack{
p\in S_2\setminus S_1 
\\
\log B<p \leq  B^{\psi(B)}  
}} (\theta_p+\sigma_p)
\ll B^{n+1}
,\]
 from which we can obtain
\[
\sum_{\substack{
x\in \P^n(\Q),
H(x) \leq B
\\
f^{-1}(x) \text{ non-singular}
}} 
| Y_{B}(t,x) - Z'_{B}(t,x)  |
\ll B^{n+1}
(\log \log B)^{-\frac{1}{2} }
.\]
An application of  Markov's inequality 
as in the last stage of the proof of Lemma~\ref{lem:tranghfa0}
concludes the proof.
\end{proof} 
\begin{remark}
\label
{rem:Apocalyptic_Triumphator}
The statement of 
Theorem~\ref{thm:main}
remains valid when 
$ X_{B}(\bullet ,x)$ is replaced by any of the functions
\[
 Y_{B}(\bullet ,x),
 Z'_{B}(\bullet ,x)
\text{  or }
 Z_{B}(\bullet ,x)
.\]
This can be seen by bringing together
Lemmas~\ref{lem:tranghfa0},~\ref{lem:English_Suite_No_2}
and~\ref{lem:Contrapunctus_XIV}.
\end{remark}

\subsection{Proof of Theorem~\ref{thm:beyondeqpartitapassacaglia}       }
\label{s:needmus} 
Letting 
\[
A:=\Big\{u\in D:z> 
\int_{0}^1
u(t)^2 \mathrm{d}t
\Big\}
\]
and
combining
Donsker's theorem with 
the result of
Erd{\H{o}}s and 
Kac~\cite[III]{MR0015705}
we obtain 
\[W(A)=\tau_2(z).\]
By Remark~\ref{rem:Apocalyptic_Triumphator}
we can use  Theorem~\ref{thm:main} with
$ X_{B}(\bullet ,x)$ replaced by 
$ Z_{B}(\bullet ,x)$.
This yields 
\[
\lim_{B\to+\infty}
\b P_B\l(x\in \Omega_B: Z_{B}(\bullet ,x) \in A \r)=
\tau_2(z).
\]
To complete the proof it remains to analyse the condition $Z_{B}(\bullet ,x) \in A$.
Labelling all primes 
in ascending order as $q_1=2,q_2=3,\ldots$,
we see that the condition is
equivalent to 
\begin{align*}
z>\int_0
^1
Z_{B}(\bullet ,x)^2
&=
\sum_{q_{i+1}
\leq B}
\l(\frac
{\omega_f(x,q_{i+1})  -    \sum_{q\leq q_{i+1}  }   \sigma_q}
{\l(    \Delta(f)   \log \log B\r)^{\frac{1}{2}} }
\r)^2
\mathrm{meas}
\l(T_i(B,x)\r)
\\&=
\frac{1}{\l(\Delta(f) \log \log B\r)^2 }
\sum_{3\leq p\leq B}
 \sigma_p 
\Big(
\omega_f(x,p)  -    \sum_{q\leq p  }   \sigma_q 
\Big)^2
,\end{align*}
by~\eqref{eq:meas} and~\eqref{eq:stable}.
This concludes the proof
because the contribution of the prime $p=2$ in the last sum is 
$O((\log \log B)^{-2})=o(1)$.
\qed

\subsection{Proof of Theorem~\ref{thm:arcsnl} } 
\label{s:needmus7}
Let us now proceed to the proof of Theorem~\ref{thm:arcsnl}.
For $0\leq \alpha \leq \beta \leq 1$  
 define 
\[A:=\{u \in D: 
\alpha<
\mathrm{meas}
(0\leq t \leq 1: u(t)>0)
\leq \beta
\}
.\]
By~\cite[Th. 5.28]{MR2604525}
we have 
\[W(A)=\frac{1}{\pi}
\int_\alpha^\beta
\frac{\mathrm{d}s}{\sqrt{s(1-s)}}
,\]
thus, by Theorem~\ref{thm:main}
and Remark~\ref{rem:Apocalyptic_Triumphator}
we obtain  
\[
\lim_{B\to+\infty}
\b P_B\l(x\in \Omega_B: Z_{B}(\bullet ,x) \in A \r)=
\frac{1}{\pi}
\int_\alpha^\beta
\frac{\mathrm{d}s}{\sqrt{s(1-s)}}
.\]
We have $Z_{B}(\bullet ,x) \in A$
if and only if  
\[ 
\alpha<
\mathrm{meas}
\l(0\leq t \leq 1: Z_{B}(\bullet ,x)>0  \r)
\leq \beta
.\]
Labelling all primes 
in ascending order as $q_1=2,q_2=3,\ldots$,
and alluding to~\eqref{eq:meas} and~\eqref{eq:stable}
we obtain
\begin{align*}
\mathrm{meas}
\l(0\leq t \leq 1: Z_{B}(\bullet ,x)
>0  \r)
&=
\sum_{\substack {   i\geq 1   \\  \omega_f(x,q_{i+1})  >   \sum_{q\leq q_{i+1}  }   \sigma_q      }}
\mathrm{meas}
\l(
T_i(B,x)
\r)
\\&=
\sum_{\substack{ i\geq 1\\ q_{i+1} \in \c C_f(x)} } 
\frac{  \sigma_{q_{i+1}}   }
{\Delta(f) \log \log B}
\\&=\frac{\widehat{\c C}_f(x)-c \sigma_2}{\Delta(f) \log \log B}
,\end{align*}  where 
the term $c$ equals $1$ if $\omega_f(x,2)>\sigma_2$ and is $0$ otherwise.
If $B^{1/2}<H(x)\leq B$ then 
$-1+
\log \log B 
\leq 
\log \log H(x)
\leq 
\log \log B
$, hence 
for $100\%$ of all $x\in \P^n(\Q)$ one has 
\[
\frac{\widehat{\c C}_f(x)-c \sigma_2}{\Delta(f) \log \log B}
=
\frac{\widehat{\c C}_f(x) }{\Delta(f) \log \log H(x)}
+O\l(\frac{1}{\log \log B}\r)
.\]
This concludes the proof of Theorem~\ref{thm:arcsnl}.
\qed
\subsection{Lower bounds for $\widehat{\c C}_f$}
\label{ex:klapaklapa} 
Let us provide   an example which shows that~\eqref{eq:tsabakafedaki} is best possible. 
Let $V$ be the conic bundle 
$x_0^2+x_1^2=st x_2^2$
and define $f:V\to \P_\Q^1$ 
through $f(x_0,x_1,x_2,s,t):=(s,t)$.
It is easy to see that $\Delta(f)=1$ and that 
\[\sigma_p=
\begin{cases} \frac{2}{p+1}, &\mbox{if } p\equiv 3\md{4},\\ 
0, & \mbox{if } p\equiv 1\md{4}.\end{cases}
\]
Label all primes 
$q\equiv 3\md{4}$ in ascending order 
by $q_1<q_2<\ldots$ and for each $N\in \N$
define 
\[x_N:=\Bigg[1:\prod_{i=1}^N q_i \Bigg] \in \P_\Q^1.\]
One can use Hilbert symbols (see~\cite[Ch.III,Th.1]{serrecourse})
to show that $$\big\{p \text{ prime} : f^{-1}(x_N)(\Q_p)=\emptyset\big\}=\big\{q_i:1\leq i \leq N\big\}.$$
Next, note that for any prime $p\leq q_N$ we have 
\[\omega_f(x_N,p)=\#\{q\equiv 3\md{4}:q\leq p\}
\sim \frac{1}{2}\frac{\log p}{\log \log p}, \text{ as } p\to+\infty,\]
due to the prime number theorem for arithmetic progressions.
Clearly this is greater than the quantity 
$  \sum_{q\leq p}1/q$ for all sufficiently large $p$, therefore $\c C_f(x)$ contains all primes  $p$ 
in the range $1\ll p \leq q_N$, with an absolute implied constant.
Letting $p'$ be the largest prime with $\log \log p'<N$
we obtain that whenever $p\in (q_N,p']$ then 
\[\omega_f(x_N,p)=
\omega_f(x_N)=N
>\log \log p' \geq 
\log \log p'
,\] therefore $\c C_f(x_N)$ contains all primes  $p$ in the range  $(q_N,p']$.
We obtain that 
\[\widehat{\c C}_f(x_N)
\geq 
\sum_{\substack{1\ll p\leq p'\\p\equiv 3\md{4}}} \frac{2}{p+1}
\gg
\log \log (p'+1)\geq N.
\]
The 
 prime number theorem for arithmetic progressions
shows that $$\log H(x_N)=
\sum_{\substack{p\leq q_N \\ p\equiv 3\md{4}}} \log p
\sim q_N \sim 2 N \log N, \text{ as } N\to+\infty,
$$ therefore 
\[\widehat{\c C}_f(x_N) \gg N \gg \frac{\log H(x_N)}{\log \log H(x_N)}\]
for all sufficiently large $N\in \N$.

\subsection{Proof of Theorem~\ref{thm:feynmkac}  } 
\label{s:feynmkac}
By  Theorem~\ref{thm:main}
and Remark~\ref{rem:Apocalyptic_Triumphator}
the random function 
$Z_{B}(\bullet ,x)$ converges in distribution to the
standard
 Wiener process.
Fix $t$ and $u$ as in the statement of  Theorem~\ref{thm:feynmkac}.
Letting 
$h:D\to \R$ be given by
\[h(u):=\exp\l(
-u \int_0^t
\c K(u(\tau))
\mathrm{d}\tau
\r)
,\]
we obtain 
\beq{eq:bwvsort1}{
\lim_{B\to+\infty}
\mathbb{E}_{x\in \Omega_B}
\l(h\l(Z_{B}(\bullet ,x)\r) \r)
=
 \mathbb{E}^0
\l(
\exp\l\{-u
\int_{0}^t
\c K(B_\tau)
\mathrm{d}\tau
\r\}
\r) 
,}
where $ \mathbb{E}^0$ is taken over 
 all Brownian motion paths $\{B_\tau:\tau\geq 0\}$ satisfying $B_0=0$ almost surely 
 and with respect to the Wiener measure $W$.
We have 
\beq{eq:bwvsort2}{
\mathbb{E}_{x\in \Omega_B}
\l(h\l(
Z_{B}(\bullet ,x)
\r) 
\r)
=
\frac{1}{\#\Omega_B}
\sum_{\substack{x\in \P^n(\Q) \\
H(x)\leq B
} }
\exp\l(
-u \int_0^t
\c K(
Z_{B}(\tau ,x)
)
\mathrm{d}\tau
\r)
}
and it thus remains to analyse the last integral.
Labelling all primes 
in ascending order as $q_1=2,q_2=3,\ldots$
and using~\eqref{eq:stable}
gives
us
\beq{eq:sumini}{\int_0^t
\c K(
Z_{B}(\tau ,x)
)
\mathrm{d}\tau
=\sum_{i\geq 1}
\c K\!
\l(
\frac
{\omega_f(x,q_{i+1})  -    \sum_{q\leq q_{i+1}  }   \sigma_q}
{\l(    \Delta(f)   \log \log B\r)^{\frac{1}{2}} }
\r)
\mathrm{meas}
\l(T_i(B,x) \cap [0,t] \r)
.}
Note that 
if \[j
(=j(t,x,B))
:=\max
\l\{
i\geq 1
:
\sum_{q\leq q_j} \sigma_q
\leq
\Delta(f)
t
\log \log B
\r\}
,\]
then the sum in~\eqref{eq:sumini}
includes all terms with $i\leq j-1$
and does not include any term with $i\geq j+1$.
Hence by~\eqref{eq:meas} the sum equals 
\[
\frac{1}{\Delta(f) \log \log B} 
\sum_{p\leq q_j
}
\sigma_p
\c K\!
\l(
\frac
{\omega_f(x, p )  -    \sum_{q\leq p  }   \sigma_q}
{\l(    \Delta(f)   \log \log B\r)^{\frac{1}{2}} }
\r)
+O\l(\frac{1}{\log \log B}\r)
,\]
where 
we have set $p=q_{i+1}$
and 
the error term 
is due to the term with   $i=j$
and the fact that $\c K$ is bounded and non-negative.
Furthermore, 
the implied constant depends at most on $f$.
The definition of $j$ 
implies that 
\[
\sum_{p\leq q_j} \sigma_p
\leq
\Delta(f)
t
\log \log B
<
\sum_{p\leq q_{j+1}} \sigma_p
 \]
and therefore by Lemma~\ref{lem:abelpu8}
there exist non-negative constants $c_0,c_1$ 
such that 
\[
-c_0
+t \log \log B
<
\log \log q_j
\leq 
c_1
+t \log \log B
.\]
Using the fact that $\c K$ is bounded
shows that 
  the difference 
\[
\sum_{p\leq q_j
}
\sigma_p
\c K\!
\l(
\frac
{\omega_f(x, p )  -    \sum_{q\leq p  }   \sigma_q}
{\l(    \Delta(f)   \log \log B\r)^{\frac{1}{2}} }
\r)
-
\sum_{p\leq \exp\l(\log^t B\r)
}
\sigma_p
\c K\!
\l(
\frac
{\omega_f(x, p )  -    \sum_{q\leq p  }   \sigma_q}
{\l(    \Delta(f)   \log \log B\r)^{\frac{1}{2}} }
\r)
\]
has modulus 
\[
\ll 
\sum\l\{
\sigma_p:\log \log p
\in
(
-c_0
+t \log \log B
, 
c_1
+t \log \log B
]
\r\}
\ll
1
,\] 
with an implied constant.
depending at most on $f$.
Recalling~\eqref{def:frenchsuites}
gives
\beq{eq:bwvsort3}{
\int_0^t
\c K(
Z_{B}(\tau ,x)
)
\mathrm{d}\tau
=\widetilde{\c K}_B(x,t)
+O\l(\frac{1}{\log \log B}\r)
,} with an implied constant depending at most on $f$. 
Combining~\eqref{eq:bwvsort1},
\eqref{eq:bwvsort2},
\eqref{eq:bwvsort3}
and
\eqref{eq:feynmankac}
concludes the proof.
\qed

	\end{document}